\documentclass[a4paper, 12pt]{article}

\usepackage[margin=2cm]{geometry}
\usepackage{hyperref,xcolor}
\usepackage{amsthm, amssymb, amsmath}
\usepackage{enumerate}
\usepackage{algorithmic}
\usepackage{authblk}

\hypersetup{
    colorlinks=false,
	linkcolor=red,
	citecolor=red,
    pdfborder={0 0 0},
}

\title{Pairwise Suitable Family of Permutations and Boxicity}

\author[1]{Manu~Basavaraju}
\author[2]{L.~Sunil~Chandran}
\author[3]{Rogers~Mathew\footnote{Supported by an AARMS Postdoctoral Fellowship}}
\author[2]{Deepak~Rajendraprasad\footnote{Partially supported by Microsoft Research India PhD Fellowship}}

\affil[1]{
	The Institute of Mathematical Sciences, \authorcr
	Chennai, India - 600113. \authorcr
	manub@imsc.res.in
}
\affil[2]{
	Department of Computer Science and Automation, \authorcr 
	Indian Institute of Science, 
    Bangalore, India - 560012. \authorcr
	\{sunil, deepakr\}@csa.iisc.ernet.in
}
\affil[3]
{
	Department of Mathematics and Statistics, \authorcr 
	Dalhousie University, 
	Halifax, Canada - B3H 3J5. \authorcr
	rogersm@mathstat.dal.ca
}

\theoremstyle{definition}
\newtheorem{definition}{Definition}

\theoremstyle{plain}
\newtheorem{theorem}{Theorem}
\newtheorem{lemma}[theorem]{Lemma}
\newtheorem{corollary}[theorem]{Corollary}
\newtheorem{conjecture}[theorem]{Conjecture}
\newtheorem{openproblem}[theorem]{Open problem}
\newtheorem{observation}[theorem]{Observation}

\theoremstyle{remark}

\newtheoremstyle{plainitshape}
  {}
  {}
  {\itshape}
  {}
  {\itshape}
  {.}
  {0.5em}
  {}
\theoremstyle{plainitshape}
\newtheorem{claim}{Claim}[theorem]
\newtheorem{property}{Property}

\newtheoremstyle{cases}
  {}
  {}
  {}
  {}
  {}
  {\newline}
  {0.5em}
  {{\itshape \thmname{#1}} \thmnumber{#2} ({\itshape\thmnote{#3}}).\medskip}
\theoremstyle{cases}
\newtheorem{case}{Case}

\newtheoremstyle{constructions}
  {}
  {}
  {}
  {}
  {}
  {}
  {0.5em}
  {{\itshape \thmname{#1}} \thmnumber{#2}\medskip}
\theoremstyle{constructions}
\newtheorem{construction}{Construction}[theorem]

\newcommand{\boxli}{\pi}
\newcommand{\boxlistar}{\boxli^{\star}}
\newcommand{\boxi}{\operatorname{boxicity}}

\newcommand{\pw}{\operatorname{pw}}
\newcommand{\tw}{\operatorname{tw}}
\newcommand{\idim}{\operatorname{dim}}

\newcommand{\leftend}{l}
\newcommand{\rightend}{r}

\newcommand{\floor}[1]{\left\lfloor #1 \right\rfloor}
\newcommand{\ceil}[1]{\left\lceil #1 \right\rceil}
\newcommand{\ilog}{\operatorname{log^{\star}}}
\newcommand{\order}[1]{O\left( #1 \right)}
\newcommand{\orderatleast}[1]{\Omega\left( #1 \right)}
\newcommand{\orderexactly}[1]{\Theta\left( #1 \right)}

\def\tends{\rightarrow}
\def\into{\rightarrow}
\def\N{\mathbb{N}}
\def\R{\mathbb{R}}

\def\G{\mathcal{G}}
\def\F{\mathcal{F}}
\def\E{\mathcal{E}}

\newcommand{\rogers}[1]{{\color{green} #1}}

\begin{document}
\maketitle

\begin{abstract}
A family $\F$ of permutations of the vertices of a hypergraph $H$ is called {\em pairwise suitable} for $H$ if, for every pair of disjoint edges in $H$, there exists a permutation in $\F$ in which all the vertices in one edge precede those in the other. The cardinality of a smallest such family of permutations for $H$ is called the {\em separation dimension} of $H$ and is denoted by $\boxli(H)$. Equivalently, $\boxli(H)$ is the smallest natural number $k$ so that the vertices of $H$ can be embedded in $\R^k$ such that any two disjoint edges of $H$ can be separated by a hyperplane normal to one of the axes. We show that the separation dimension of a hypergraph $H$ is equal to the {\em boxicity} of the line graph of $H$. This connection helps us in borrowing results and techniques from the extensive literature on boxicity to study the concept of separation dimension. 


\vspace{1ex}
\end{abstract}


\section{Introduction}

Let $\sigma:U \into [n]$ be a permutation of  elements of an $n$-set $U$. For two disjoint subsets $A,B$ of $U$, we say  $A \prec_{\sigma} B$ when every element of $A$ precedes every element of $B$ in $\sigma$, i.e., $\sigma(a) < \sigma(b), \forall (a,b) \in A \times B$. Otherwise, we say $A \nprec_{\sigma} B$. We say that $\sigma$ {\em separates} $A$ and $B$ if either $A \prec_{\sigma} B$ or $B \prec_{\sigma} A$. We use $a \prec_{\sigma} b$ to denote $\{a\} \prec_{\sigma} \{b\}$. For two subsets $A, B$ of $U$, we say $A \preceq_{\sigma} B$ when $A \setminus B \prec_{\sigma} A\cap B \prec_{\sigma} B \setminus A$. 

In this paper, we introduce and study a notion called \emph{pairwise suitable family of permutations} for a hypergraph $H$.

\begin{definition}
\label{definitionPairwiseSuitable}
A family $\F$ of permutations of $V(H)$ is \emph{pairwise suitable} for a hypergraph $H$ if, for every two disjoint edges $e,f \in E(H)$, there exists a permutation $\sigma \in \F$  which separates $e$ and $f$. The cardinality of a smallest family of permutations that is pairwise suitable for $H$ is called the {\em separation dimension} of $H$ and is denoted by $\boxli(H)$. 
\end{definition}

A family $\F = \{\sigma_1, \ldots, \sigma_k\}$ of permutations of a set $V$ can be seen as an embedding of $V$ into $\R^k$ with the $i$-th coordinate of $v \in V$ being the rank of $v$ in the $\sigma_i$. Similarly, given any embedding of $V$ in $\R^k$, we can construct $k$ permutations by projecting the points onto each of the $k$ axes and then reading them along the axis, breaking the ties arbitrarily. From this, it is easy to see that $\boxli(H)$ is the smallest natural number $k$ so that the vertices of $H$ can be embedded into $\R^k$ such that any two disjoint edges of $H$ can be separated by a hyperplane normal to one of the axes. This motivates us to call such an embedding a {\em separating embedding} of $H$ and $\boxli(H)$ the {\em separation dimension} of $H$.

The study of similar families of permutations dates back to the work of Ben Dushnik in 1947 where he introduced the notion of \emph{$k$-suitability} \cite{dushnik}. A family $\F$ of permutations of $[n]$ is \emph{$k$-suitable} if, for every $k$-set $A \subseteq [n]$ and for every $a \in A$, there exists a $\sigma \in \mathcal{F}$ such that $A \preceq_{\sigma} \{a\}$. Let $N(n,k)$ denote the cardinality of a smallest family of permutations that is $k$-suitable for $[n]$. In 1971, Spencer \cite{scramble} proved that $\log \log n \leq N(n,3) \leq N(n,k) \leq k2^k\log \log n$. He also showed that $N(n,3) < \log \log n + \frac{1}{2} \log \log \log n + \log (\sqrt{2}\boxli) + o(1)$. Fishburn and Trotter, in 1992, defined the \emph{dimension} of a hypergraph on the vertex set $[n]$ to be the minimum size of a family $\F$ of permutations of $[n]$ such that every edge of the hypergraph is an intersection of \emph{initial segments} of $\F$ \cite{FishburnTrotter1992}. It is easy to see that an edge $e$ is an intersection of initial segments of $\F$ if and only if for every $v \in [n] \setminus e$, there exists a permutation $\sigma \in \F$ such that $e \prec_{\sigma} \{v\}$. F\"{u}redi, in 1996, studied the notion of \emph{$3$-mixing} family of permutations \cite{Furedi1996}. A family $\F$ of permutations of $[n]$ is called $3$-mixing if for every $3$-set $\{a, b, c\} \subseteq [n]$ and a designated element $a$ in that set, one of the permutations in $\F$ places the element $a$ between $b$ and $c$. It is clear that $a$ is between $b$ and $c$ in a permutation $\sigma$ if and only if $\{a,b\} \preceq_{\sigma} \{a,c\}$ or $\{a.c\} \preceq_{\sigma} \{a,b\}$. Such families of permutations with small sizes have found applications in showing upper bounds for many combinatorial parameters like poset dimension \cite{kierstead1996order}, product dimension \cite{FurediPrague}, boxicity \cite{RogSunSiv} etc.

The notion of separation dimension introduced here seems so natural but, to the best of our knowledge, has not been studied in this generality before. Apart from that, a major motivation for us to study this notion of separation is its interesting connection with a certain well studied geometric representation of graphs.  In fact, we show that $\boxli(H)$ is same as the \emph{boxicity} of the intersection graph of the edge set of $H$, i.e., the line graph of $H$.
  
An axis-parallel $k$-dimensional box or a \emph{$k$-box} is  a Cartesian  product $R_1 \times  \cdots \times R_k$, where each  $R_i$ is a closed interval on the real line. For example, a line segment lying parallel to the $X$ axis is a $1$-box, a rectangle with its sides parallel to the $X$ and $Y$ axes is a $2$-box, a rectangular cuboid with its sides parallel to the $X$, $Y$, and $Z$ axes is a $3$-box and so on. A {\em box representation} of a graph $G$ is a geometric representation of $G$ using axis-parallel boxes as follows. 

\begin{definition}
\label{definitionBoxicity}
The \emph{$k$-box representation} of a graph $G$ is a function $f$ that maps each vertex in $G$ to a $k$-box in $\mathbb{R}^k$ such that, for all vertices $u,v$ in $G$, the pair $\{u,v\}$ is an edge if and only if $f(u)$ intersects $f(v)$. The \emph{boxicity} of a graph $G$, denoted by $\boxi(G)$, is the minimum positive integer $k$ such that $G$ has a $k$-box representation. 
\end{definition}

Box representation is a generalisation of interval representation of \emph{interval graphs} (intersection graphs of closed intervals on the real line). From the definition of boxicity, it is easy to see that interval graphs are precisely the graphs with boxicity $1$.  The concept of boxicity was introduced by F.S. Roberts in 1969 \cite{Roberts}. He showed that every graph on $n$ vertices has an $\floor{ n/2  }$-box representation. The $n$-vertex graph whose complement is a perfect matching is an example of a graph whose boxicity is equal to $n/2$. Upper bounds for boxicity in terms of other graph parameters like maximum degree, treewidth, minimum vertex cover, degeneracy etc. are available in literature. Adiga, Bhowmick, and Chandran showed that the boxicity of a graph with maximum degree $\Delta$ is $O(\Delta \log^2 \Delta)$ \cite{DiptAdiga}. Chandran and Sivadasan proved that boxicity of a graph with treewidth $t$ is at most $t+2$ \cite{CN05}. It was shown by Adiga, Chandran and Mathew that the boxicity of a $k$-degenerate graph on $n$ vertices is $O(k \log n)$ \cite{RogSunAbh}. Boxicity is also studied in relation with other dimensional parameters of graphs like partial order dimension and threshold dimension \cite{DiptAdiga,Yan1}. Studies on box representations of special graph classes too are available in abundance. Scheinerman showed that every outerplanar graph has a $2$-box representation \cite{Scheiner} while Thomassen showed that every planar graph has a $3$-box representation \cite{Thoma1}. Results on boxicity of series-parallel graphs \cite{CRB1}, Halin graphs \cite{halinbox}, chordal graphs, AT-free graphs, permutation graphs \cite{CN05}, circular arc graphs \cite{Dipt}, chordal bipartite graphs \cite{SunMatRog} etc. can be found in literature. Here we are interested in boxicity of the line graph of hypergraphs.

\begin{definition}
\label{definitionLineGraph}
The \emph{line graph} of a hypergraph $H$, denoted by $L(H)$, is the graph with vertex set $V(L(H)) = E(H)$ and edge set $E(L(H)) = \{\{e,f\} : e, f \in E(H), e \cap f \neq \emptyset \}$.
\end{definition}
  
For the line graph of a graph $G$ with maximum degree $\Delta$, it was shown by Chandran, Mathew and Sivadasan that its boxicity is $\order{\Delta \log\log \Delta}$ \cite{RogSunSiv}. It was in their attempt to improve this result that the authors stumbled upon pairwise suitable family of permutations and its relation with the boxicity of the line graph of $G$. May be we should mention in passing that though line graphs of graphs form a proper subclass of graphs, any graph is a line graph of some hypergraph.


\subsection{Summary of results}
\newcounter{counterResult}

Some of the results in this paper are interesting because of their consequences. Some are interesting because of the connections with other questions in combinatorics, which are exploited to good effect in their proof. Hence, in this section summarising our results, we indicate those connections along with the consequences. The definitions of the parameters mentioned are given in the appropriate sections. As noted earlier, the motivating result for this paper is the following: 
\begin{enumerate}
\item For any hypergraph $H$, $\boxli(H)$ is precisely the boxicity of the line graph of $H$, i.e., 
$$\boxli(H) = \boxi(L(H)) 
	\mbox{\hspace{5ex} (Theorem \ref{theoremConnectionBoxliPermutation})}.$$
\setcounter{counterResult}{\value{enumi}}
\end{enumerate}

It is the discovery of this intriguing connection that aroused our interest in the study of pairwise suitable families of permutations. This immediately makes applicable every result in the area of boxicity to separation dimension. For example, any hypergraph with $m$ edges can be separated in $\R^{\floor{m/2}}$; for every $m \in \N$, there exist hypergraphs with $m$ edges which cannot be separated in any proper subspace of $\R^{\floor{m/2}}$; every hypergraph whose line graph is planar can be separated in $\R^3$; every hypergraph whose line graph has a treewidth at most $t$ can be separated in $\R^{t+2}$; hypergraphs separable in $\R^1$ are precisely those whose line graphs are interval graphs and so on. Further, algorithmic and hardness results from boxicity carry over to separation dimension since constructing the line graph of a hypergraph can be done in quadratic time. We just mention two of them. Deciding if the separation dimension is at most $k$ is NP-Complete for every $k \geq 2$ \cite{Coz,Kratochvil} and unless NP = ZPP, for any $\epsilon >0$, there does not exist a polynomial time algorithm to approximate the separation dimension of a hypergraph within a factor of $m^{1/2 -\epsilon}$ where $m = |E(H)|$ \cite{DiptAdigaHardness} \footnote{A recent preprint claims that the inapproximability factor can be improved to $m^{1 - \epsilon}$, which is essentially tight. \cite{chalermsookgraph}.}. In this work, we have tried to find bounds on the separation dimension of a hypergraph in terms of natural invariants of the hypergraph like maximum degree, rank etc. The next two results are for rank-$r$ hypergraphs. 

\begin{enumerate}
\setcounter{enumi}{\value{counterResult}}

\item 
For any rank-$r$ hypergraph $H$ on $n$ vertices 
$$\boxli(H) \leq \frac{e \ln 2}{\pi \sqrt{2}}  4^r \sqrt{r} \log n 
	\mbox{\hspace{5ex} (Theorem \ref{theoremHypergraphSizeUpperbound})}.
$$
The bound is obtained by direct probabilistic arguments. The next result shows that this bound is tight up to a factor of constant times $r$.
\item
Let $K_n^r$ denote the complete $r$-uniform hypergraph on $n$ vertices with $r > 2$. Then 
$$ c_1 \frac{4^r}{\sqrt{r-2}} \log n \leq \boxli(K_n^r) \leq c_2 4^r \sqrt{r} \log n,$$ for $n$ sufficiently larger than $r$ and where $c_1 = \frac{1}{2^7}$ and  $c_2 =  \frac{e\ln2}{\pi\sqrt{2}} < \frac{1}{2}$ (Theorem \ref{theoremHypergraphSizeLowerbound}). The lower bound is obtained by first proving that the separation dimension of $K_n$, the complete graph on $n$ vertices, is in $\orderatleast{\log n}$ and then showing that, given any separating embedding of $K_n^r$ in $\R^d$, the space $\R^d$ contains ${2r-4 \choose r-2}$ orthogonal subspaces such that the projection of the given embedding on to these subspaces gives a separating embedding of a $K_{n-2r+4}$.   

\item
For any rank-$r$ hypergraph $H$ of maximum degree $D$,
$$\boxli(H) \leq \order{rD \log^2(rD)} 
	\mbox{\hspace{5ex} (Corollary \ref{corollaryHypergraphMaxDegree})}. 
$$
This is a direct consequence of the nontrivial fact that $\boxi(G) \in \order{\Delta \log^2 \Delta}$ for any graph $G$ of maximum degree $\Delta$ \cite{DiptAdiga}. Further using the fact, again a nontrivial one,  that there exist graphs of maximum degree $\Delta$ with boxicity $\orderatleast{\Delta \log \Delta}$ \cite{DiptAdiga}, we show that there exists rank-$r$ hypergraphs of maximum degree $2$ with separation dimension in $\orderatleast{r \log r}$. It is trivial to see that the separation dimension of hypergraphs with maximum degree $1$ cannot be more than $1$.

\setcounter{counterResult}{\value{enumi}}
\end{enumerate}

Below we highlight the main results in this paper when we restrict $H$ to be a graph. Every graph has a non-crossing straight line 3D drawing, which is nothing but an embedding of the vertices of a graph into $\R^3$ such that  any two disjoint edges can be separated by a plane. Hence if we allow separating hyperplanes of all orientations, then we can have a separating embedding of every graph in $\R^3$. But if we demand that all the separating hyperplanes be normal to one of the coordinate axes, then the story changes. 

 
\vspace{1ex}
\noindent For a graph $G$ on $n$ vertices, we show the following upper bounds. 

\begin{enumerate}
\setcounter{enumi}{\value{counterResult}}

\item $\boxli(G) \leq 6.84 \log n$ (Theorem \ref{theoremBoxliSize}).
This bound is obtained by simple probabilistic arguments. We also prove that this bound is tight up to constant factors by showing that a complete graph $K_n$ on $n$ vertices has $\boxli(K_n) \geq \log \floor{n/2}$.
 
\item $\boxli(G) \leq 2^{9 \ilog \Delta} \Delta$, where $\Delta$ denotes the maximum degree of $G$ (Theorem \ref{theoremBoxliDelta}).
This is an improvement over the upper bound of $\order{\Delta \log \log \Delta}$ for the boxicity of the line graph of $G$ proved in \cite{RogSunSiv}.
The proof technique works by recursively partitioning the graph into $\order{\Delta / \log \Delta}$ parts such that no vertex has more than $\frac{1}{2} \log \Delta$ neighbours in any part and then attacking all possible pairs of these parts..

\item $\boxli(G) \in O(k \log \log n)$, where $k$ is the degeneracy of $G$ (Theorem \ref{theoremBoxliDegeneracy}).
This is proved by decomposing $G$ into $2k$ star forests and using $3$-suitable permutations of the stars in every forest and the leaves in every such star simultaneously. We also show that the $\log \log n$ factor in this bound cannot be improved in general by demonstrating that for the fully subdivided clique $K_n^{1/2}$, which is a $2$-degenerate graph, $\boxli(K_n^{1/2}) \in \orderexactly{\log \log n}$.

\item $\boxli(G) \in \order{\log(t+1)}$, where $t$ denotes the treewidth of $G$ (Theorem \ref{theoremBoxliTreewidth}). This is proved by adjoining a family of pairwise suitable permutations of the colour classes of a minimal chordal supergraph of $G$ with $2\log (t+1)$ more ``colour sensitive'' permutations based on a DFS traversal of the tree. This bound is also seen to be tight up to constant factors because the clique $K_n$, whose treewidth is $n-1$, has  $\boxli(K_n) \geq \log \floor{n/2}$. 

\item $\boxli(G) \leq 2\chi_a + 13.68\log \chi_a$ and $\boxli(G) \leq \chi_s + 13.68\log \chi_s$, where $\chi_a$ and $\chi_s$ denote, respectively, the acyclic chromatic number and star chromatic number of $G$ (Theorem \ref{theoremBoxliAcyclicStar}). Both the bounds are obtained by exploiting the structure of the graph induced on a pair of colour classes. This bound, when combined with certain results from literature immediately gives a few more upper bounds (Corollary \ref{corollaryBoxliGenus}):
(i)   $\boxli(G) \in O(g^{4/7})$, where $g$ is the Euler genus of $G$; and
(ii)  $\boxli(G) \in O(t^2 \log t)$, if $G$ has no $K_t$ minor.

\item $\boxli(G) \leq 3$, if $G$ is planar (Theorem \ref{theoremBoxliPlanar}). This is proved using Schnyder's celebrated result on planar drawing \cite{schnyder1990embedding}. This bound is the best possible since the separation dimension of $K_4$ is $3$. 

\item $\boxli(G^{1/2}) \leq (1 + o(1)) \log \log (\chi - 1) + 2$, where $G^{1/2}$ is the graph obtained by subdividing every edge of  $G$ and $\chi$ is the chromatic number of $G$ (Corollary \ref{corollarySubdivisionChromaticNumber}). This is proved by associating with every graph $G$ an interval order whose dimension is at least $\boxli(G^{1/2})$ and whose height is less than the chromatic number of $G$. The tightness, up to a factor of $2$, of the above bound follows from our result that $\boxli(K_n^{1/2}) \geq \frac{1}{2} \floor{\log\log (n-1)}$.

\item 
If $G$ is the $d$-dimensional hypercube $Q_d$, 
$$ \frac{1}{2} \floor{\log\log(d-1)} \leq 
		\boxli(Q_d) \leq
		(1 + o(1)) \log\log d,
$$ 
where $c$ is a constant (Theorem \ref{theoremHypercube}). The lower bound follows since $K_d^{1/2}$ is contained as a subgraph of $Q_d$. The upper bound is obtained by taking a $3$-suitable family of permutations of the $d$ positions of the binary strings and defining an order on the strings itself based on that.

\setcounter{counterResult}{\value{enumi}}
\end{enumerate}

The main lower bounding strategy that we employ in this paper is the following result that we prove in Theorem \ref{theoremBoxliLowerBound}.

\begin{enumerate}
\setcounter{enumi}{\value{counterResult}}
\item
For a graph $G$, let $V_1, V_2 \subsetneq V(G)$ such that $V_1 \cap V_2 = \emptyset$. If  there exists an edge between every $s_1$-subset of $V_1$ and every $s_2$-subset of $V_2$, then $\boxli(G) \geq \min \left\{ \log \frac{|V_1|}{s_1}, \log \frac{|V_2|}{s_2} \right\}$.  

This immediately shows that $\boxli(K_{n,n}) \geq \log n$, $\boxli(K_n) \geq \log \floor{n/2}$ and that for any graph $G$, $\boxli(G) \geq \log \floor{\omega/2}$, where  $\omega$ denotes the size of a largest clique in $G$. It also forms a key ingredient in showing the lower bound on separation dimension of the complete $r$-uniform hypergraph. Finally it is used to derive the following lower bound for random graphs.

\item For a graph $G \in \G(n,p)$, $\boxli(G) \geq \log(np) - \log \log(np) - 2.5$ asymptotically almost surely (Theorem \ref{theoremBoxliLowerBoundRandom}).
\setcounter{counterResult}{\value{enumi}}
\end{enumerate}

The last result in the paper is the following lower bound on the separation dimension of fully subdivided cliques (Theorem \ref{theoremKnHalf}).
 
\begin{enumerate}
\setcounter{enumi}{\value{counterResult}}

\item Let $K_n^{1/2}$ denote the graph obtained by subdividing every edge of $K_n$ exactly once. Then, 
$$\boxli(K_n^{1/2}) \geq \frac{1}{2} \floor{\log\log(n-1) }.$$
This is proved by using Erd\H{os}-Szekeres Theorem to extract a large enough set of vertices of the underlying $K_n$ that are ordered essentially the same by every permutation in the selected family and then showing that separating the edges incident on those vertices can be modelled as a problem of finding a realiser for a canonical open interval order of same size. This lower bound is used to show the tightness of two of the upper bounds above.

\setcounter{counterResult}{\value{enumi}}
\end{enumerate}


\subsection{Outline of the paper}

The remainder of this paper is organised as follows. A brief note on some standard terms and notations used throughout this paper is given in Section \ref{sectionNotation}. Section \ref{sectionConnection} demonstrates the equivalence of separation dimension of a hypergraph $H$ and boxicity of the line graph of $H$. All the upper bounds are stated and proved in Section \ref{sectionUpperBounds}. The tightness of the upper bounds, where we know them, are mentioned alongside the bound but their proofs and discussion are postponed till the subsequent section (Section \ref{sectionLowerBound}). Finally, in Section \ref{sectionOpenProblems}, we conclude with a discussion on a few open problems that we find interesting.



\subsection{Notational note}
\label{sectionNotation}

A {\em hypergraph} $H$ is a pair $(V, E)$ where $V$, called the {\em vertex set}, is any set and $E$, called the {\em edge set}, is a collection of subsets of $V$. The vertex set and edge set of a hypergraph $H$ are denoted respectively by $V(H)$ and $E(H)$. The {\em rank} of a hypergraph $H$ is $\max_{e \in E(H)}|e|$ and $H$ is called {\em $k$-uniform} if $|e| = k, \forall e \in E(H)$. The {\em degree} of a vertex $v$ in $H$ is the number of edges of $H$ which contain $v$. The {\em maximum degree} of $H$, denoted as $\Delta(H)$ is the maximum degree over all vertices of $H$. All the hypergraphs considered in this paper are finite. 

A {\em graph} is a $2$-uniform hypergraph. For a graph $G$ and any $S \subseteq V(G)$, the subgraph of $G$ induced on the vertex set $S$ is denoted by $G[S]$. For any $v \in V(G)$, we use $N_G(v)$ to denote the neighbourhood of $v$ in $G$, i.e., $N_G(v) = \{u \in V(G) : \{v,u\} \in E(G)\}$.

A \emph{closed interval} on the real line, denoted as $[i,j]$ where $i,j \in \R$ and $i\leq j$, is the set $\{x\in \R : i\leq x\leq j\}$. Given an interval $X=[i,j]$, define $\leftend(X)=i$ and $\rightend(X)=j$. We say that the closed interval $X$ has \emph{left end-point} $\leftend(X)$ and \emph{right end-point} $\rightend(X)$. For any two intervals $[i_1, j_1], [i_2,j_2]$ on the real line, we say that $[i_1, j_1] < [i_2,j_2]$ if $j_1 < i_2$. 

For any finite positive integer $n$, we shall use $[n]$ to denote the set $\{1,\ldots , n\}$. A permutation of a finite set $V$ is a bijection from $V$ to $[|V|]$.

The logarithm of any positive real number $x$ to the base $2$ and $e$ are respectively denoted by $\log(x)$ and $\ln(x)$, while $\ilog(x)$ denotes the iterated logarithm of $x$ to the base $2$, i.e. the number of times the logarithm function (to the base $2$) should be applied so that the result is less than or equal to $1$.


\section{Pairwise suitable family of permutations and a box representation}
\label{sectionConnection}

In this section we show that a family of permutations of cardinality $k$ is pairwise suitable for a hypergraph $H$ (Definition \ref{definitionPairwiseSuitable}) if and only if the line graph of $H$ (Definition \ref{definitionLineGraph}) has a $k$-box representation (Definition \ref{definitionBoxicity}). Before we proceed to prove it, let us state an equivalent but more combinatorial definition for boxicity. 

We have already noted that interval graphs are precisely the graphs with boxicity $1$. Given a $k$-box representation of a graph $G$, orthogonally projecting the $k$-boxes to each of the $k$-axes in $\mathbb{R}^k$  gives $k$ families of intervals. Each one of these families can be thought of as an interval representation of some interval graph. Thus we get $k$ interval graphs. It is not difficult to observe that a pair of vertices is  adjacent in $G$ if and only if the pair is adjacent in each of the $k$ interval graphs obtained. The following lemma, due to Roberts \cite{Roberts}, formalises this relation between box representations and interval graphs.

\begin{lemma}[Roberts \cite{Roberts}]
\label{lemmaRoberts}
For every graph $G$, $\boxi(G) \leq k$ if and only if there exist $k$ interval graphs $I_1, \ldots, I_k$, with $V(I_1) = \cdots = V(I_k) = V(G)$ such that $G = I_1 \cap \cdots \cap I_k$.
\end{lemma}
From the above lemma, we get an equivalent definition of boxicity. 

\begin{definition}
\label{definitionBoxicityInterval}
The \emph{boxicity} of a 
graph $G$ is the minimum positive integer $k$ for which there exist $k$ interval graphs
$I_1,\ldots, I_k$ such that $G = I_1 \cap \cdots \cap I_k$.  
\end{definition}

Note that if $G = I_1 \cap \cdots \cap I_k$, then each $I_i$ is a supergraph of $G$. Moreover, for every pair of vertices $u,v \in V(G)$ with $\{u,v\} \notin E(G)$, there exists some $i \in [k]$ such that $\{u,v\} \notin E(I_i)$. Now we are ready to prove the main theorem of this section.

\begin{theorem}
For a hypergraph $H$, $\boxli(H) = \boxi(L(H))$.  
\label{theoremConnectionBoxliPermutation}
\end{theorem}
\begin{proof}
First we show that $\boxli(H) \leq \boxi(L(H))$. Let $\boxi(L(H)) = b$. Then, by Lemma \ref{lemmaRoberts}, there exists a collection of $b$ interval graphs, say $\mathcal{I} = \{I_1, \ldots, I_b\}$, whose intersection is $L(H)$. For each $i \in [b]$, let $f_i$ be an interval representation of $I_i$. For each $u \in V(H)$, let $E_H(u) = \{e \in E(H) : u \in e\}$ be the set of edges of $H$ containing $u$. Consider an $i \in [b]$ and a vertex  $u \in V(H)$. The closed interval $C_i(u) = \bigcap_{e \in E_H(u)} f_i(e)$ is called the {\em clique region} of $u$ in $f_i$. Since any two edges in $E_H(u)$ are adjacent in $L(H)$, the corresponding intervals have non-empty intersection in $f_i$ . By the Helly property of intervals, $C_i(u)$ is non-empty. We define a permutation $\sigma_i$ of $V(H)$ from $f_i$ such that $\forall u,v \in V(G)$, $C_i(u) < C_i(v) \implies u \prec_{\sigma_i} v$. It suffices to prove that $\{\sigma_1, \ldots, \sigma_b\}$ is a family of permutations that is pairwise suitable for $H$.

Consider two disjoint edges $e, e'$ in $H$. Hence $\{e, e' \} \notin E(L(H))$ and since $L(H) = \bigcap_{i=1}^b I_i$, there exists an interval graph, say $I_i \in \mathcal{I}$, such that $\{e, e'\} \notin E(I_i)$, i.e., $f_i(e) \cap f_i(e') = \emptyset$. Without loss of generality, assume $f_i(e) < f_i(e')$. For any $v \in e$ and any $v' \in e'$, since $C_i(v) \subseteq f_i(e)$ and $C_i(v') \subseteq f(e')$, we have $C_i(v) < C_i(v')$, i.e. $v \prec_{\sigma_i} v'$. Hence $e \prec_{\sigma_i} e'$. Thus the family $\{ \sigma_1, \ldots, \sigma_b \}$ of permutations is pairwise suitable for $H$. 

Next we show that $\boxi(L(H)) \leq \boxli(H)$. Let $\boxli(H) = p$ and let $\F = \{\sigma_1, \ldots, \sigma_p\}$ be a pairwise suitable family of permutations for $H$. From each permutation $\sigma_i$, we shall construct an interval graph $I_i$ such that $L(H) = \bigcap_{i=1}^{p}I_i$. Then by Lemma \ref{lemmaRoberts}, $\boxi(L(H)) \leq \boxli(H)$. 

For a given $i \in [p]$, to each edge $e \in E(H)$, we associate the closed interval 
$$f_i(e) = \left[ \min_{v \in e}\sigma_i(v) ~,~ \max_{v \in e}\sigma_i(v) \right],$$ 
and let $I_i$ be the intersection graph of the intervals $f_i(e), e \in E(H)$. Let $e, e' \in V(L(H))$. If $e$ and  $e'$ are adjacent in $L(H)$, let $v \in e \cap e'$. Then $\sigma_i(v) \in f_i(e) \cap f_i(e'),~\forall i \in [p]$. Hence $e$ and $e'$ are adjacent in $I_i$ for every $i \in [p]$. If $e$ and $e'$ are not adjacent in $L(H)$, then there is a permutation $\sigma_i \in \F$ such that either $e \prec_{\sigma_i} e'$ or $e' \prec_{\sigma_i} e$. Hence by construction $f_i(e) \cap f_i(e') = \emptyset$ and so $e$ and $e'$ are not adjacent in $I_i$. This completes the proof.
\end{proof}


\section{Upper bounds}
\label{sectionUpperBounds}

For graphs, sometimes we work with a notion of suitability that is stronger than the pairwise suitability of Definition \ref{definitionPairwiseSuitable}. This will facilitate easy proofs for some results to come later in this article. 


\begin{definition}
\label{definition3Mixing}
For a graph $G$, a family $\F$ of permutations of $G$ is \emph{$3$-mixing} if, for every two adjacent edges $\{a,b\}, \{a,c\} \in E(G)$, there exists a permutation $\sigma \in \F$ such that either $b \prec_{\sigma} a \prec_{\sigma} c$ or $c \prec_{\sigma} a \prec_{\sigma} b$. 
\end{definition}

Notice that a family of permutations $\mathcal{F}$ of $V(G)$ is pairwise suitable and $3$-mixing for $G$ if, for every two edges $e,f \in E(G)$, there exists a permutation $\sigma \in \mathcal{F}$ such that either $e \preceq_{\sigma} f$ or $f \preceq_{\sigma} e$.  Let $\boxlistar(G)$ denote the cardinality of a smallest family of permutations that is pairwise suitable and $3$-mixing for $G$. From their definitions, $\boxli(G) \leq \boxlistar(G)$.

We begin with the following two straightforward observations. 

\begin{observation}
\label{observationMonotonicity}
$\boxli(G)$ and $\boxlistar(G)$ are monotone increasing properties, i.e., $\boxli(G') \leq \boxli(G)$ and $\boxlistar(G') \leq \boxlistar(G)$ for every subgraph $G'$ of $G$.  
\end{observation}

\begin{observation}
\label{observationDisjointComponents}
Let $G_1, \ldots, G_r$ be a collection of disjoint components that form a graph $G$, i.e, $V(G) = \biguplus_{i=1}^r V(G_i)$ and $E(G) = \biguplus_{i=1}^r E(G_i)$. If $\boxli(G) \geq 1$ for some $i \in [r]$, then $\boxli(G) = \max_{i \in [r]} \boxli(G_i)$. 
\end{observation}

A nontrivial generalisation of Observation \ref{observationDisjointComponents}, when there are edges across the parts, is given in Lemma \ref{lemmaMaxPairsParts}. Now we show an upper bound on $\boxli(G)$ in terms of $|V(G)|$.


\subsection{Separation dimension and size of a hypergraph}
\label{sectionBoxliSize}


\begin{theorem}
\label{theoremHypergraphSizeUpperbound}
For any rank-$r$ hypergraph $H$ on $n$ vertices 
$$ \boxli(H) \leq \frac{e \ln 2}{\pi \sqrt{2}} 4^r \sqrt{r} \log n.$$
\end{theorem}
\begin{proof}
Consider family $\F$ of $m$ permutations of $[n]$ chosen independently and uniformly from the $n!$ possible ones. For an arbitrary pair of disjoint edges $e, f \in E(H)$, the probability $q$ that $e$ and $f$ are separated in $\sigma$ is at least $2 (r!)^2 / (2r)!$. Using Stirling's bounds $\sqrt{2\pi} k^{k + 1/2} e^{-k} \leq k! \leq e k^{k + 1/2} e^{-k}$, we get $q \geq \frac{2\pi\sqrt{2}}{e} \sqrt{r} / {4^r}$. The probability of the (bad) event that $e$ and $f$ are not separated in any of the $m$ permutations in $\F$ is at most $(1 - q)^m$. Since the number of non-empty edges in $H$ is less than $n^r$, by the union bound, the probability $p$ that there exists some pair of edges which is not separated in any of the permutations in $\F$ is less than $n^{2r}(1-q)^r \leq e^{2r \ln n}e^{-qm}$. Hence if $2r \ln n \leq qm$, then $p < 1$ and there will exist some  family $\F$ of size $m$ such that every pair of edges is separated by some permutation in $\F$. So $m \geq \frac{2r}{q} \ln n$ suffices. So $\boxli(H) \leq \frac{e}{\pi\sqrt{2}} 4^r \sqrt{r} \ln n$. 
\end{proof}


\subsubsection*{Tightness of Theorem \ref{theoremHypergraphSizeUpperbound}}

Let $K_n^r$ denote a complete $r$-uniform graph on $n$ vertices. Then by Theorem \ref{theoremHypergraphSizeLowerbound}, $\boxli(K_n^r) \geq \frac{1}{2^7} \frac{4^r}{\sqrt{r-2}} \log n$ for $n$ sufficiently larger than $r$. Hence the bound in Theorem \ref{theoremHypergraphSizeUpperbound} is tight by factor of $64 r$.

\begin{theorem}
\label{theoremBoxliSize}
For a graph $G$ on $n$ vertices,  $\boxli(G) \leq \boxlistar(G) \leq 6.84 \log n$. 
\end{theorem}
\begin{proof}
From the definitions of $\boxli(G)$ and $\boxlistar(G)$ and Observation \ref{observationMonotonicity}, we have $\boxli(G) \leq \boxlistar(G) \leq \boxlistar(K_n)$, where $K_n$ denotes the complete graph on $n$ vertices. Here we prove that  $\boxlistar(K_n) \leq 6.84 \log n$. 

Choose $r$ permutations, $\sigma_1, \ldots, \sigma_r$, independently and uniformly at random from the $n!$ distinct permutations of $[n]$. Let $e$, $f$ be two distinct edges of $K_n$. The probability that $e \preceq_{\sigma_i} f$ is $1/6$ for each $i \in [r]$. ($4$ out of $4!$ outcomes are favourable when $e$ and $f$ are non-adjacent and $1$ out of $3!$ outcomes is favourable otherwise.) 
 
\begin{eqnarray*}
Pr[(e \preceq_{\sigma_i} f)~or~(f \preceq_{\sigma_i} e)] & = & Pr[(e \preceq_{\sigma_i} f)] + Pr[(f \preceq_{\sigma_i} e)] \\
& = & \frac{1}{6} + \frac{1}{6} \\
& = & \frac{1}{3}
\end{eqnarray*}
Therefore, 
\begin{eqnarray*}
Pr[\bigcap_{i=1}^r\left((e \npreceq_{\sigma_i} f) \cap (f \npreceq_{\sigma_i} e)\right)] & = & \left(Pr[(e \npreceq_{\sigma_i} f) \cap (f \npreceq_{\sigma_i} e)]\right)^r \\
& = & (1-\frac{1}{3})^r \\
& = & \left(\frac{2}{3}\right)^r 
\end{eqnarray*}
\begin{eqnarray*}
Pr[\bigcup_{\forall\mbox{ pairs of distinct edges  }e,f}\left(\bigcap_{i=1}^r\left((e \nprec_{\sigma_i} f) \cap (f \nprec_{\sigma_i} e)\right)\right)] < n^4 \left(\frac{2}{3}\right)^r 
\end{eqnarray*}
Substituting for $r= 6.84\log n$ in the above inequality, we get 
\begin{eqnarray*}
Pr[\bigcup_{\forall\mbox{ pairs of distinct edges  }e,f}\left(\bigcap_{i=1}^r\left((e \nprec_{\sigma_i} f) \cap (f \nprec_{\sigma_i} e)\right)\right)] < 1
\end{eqnarray*}
That is, there exists a family of permutations of $V(K_n)$ of cardinality at most $6.84\log n$ which is pairwise suitable and $3$ mixing for $K_n$.  
\end{proof}

\subsubsection*{Tightness of Theorem \ref{theoremBoxliSize}}

Let $K_n$ denote a complete graph on $n$ vertices. Since $\omega(K_n) = n$, it follows from Corollary \ref{corollaryBoxliOmega} that $\boxli(K_n) \geq \log \floor{n/2}$. Hence the bound proved in Theorem \ref{theoremBoxliSize} is tight up to a constant factor.  

\subsubsection*{An auxiliary lemma}
Using Theorem \ref{theoremBoxliSize}, we shall now prove a lemma that will be used later in proving bounds for $\boxli(G)$ in terms of maximum degree, star chromatic number, and acyclic chromatic number. 

\begin{lemma}
\label{lemmaMaxPairsParts}
Let $P_G=\{V_1, \ldots, V_r\}$ be a partitioning of the vertices of a graph $G$, i.e., $V(G) = V_1 \uplus \cdots \uplus V_r$. Let $\hat{\boxli}(P_G) = \max_{i,j \in [r]} \boxli(G[V_i \cup V_j])$. Then, $\boxli(G) \leq 13.68 \log r + \hat{\boxli}(P_G) r$.   
\end{lemma}
\begin{proof}
Let $H$ be a complete graph with $V(H) = \{h_1, \ldots , h_r\}$. Let $\mathcal{M} = \{M_1, \ldots ,M_r \}$ be a collection of matchings of $H$ such that each edge is present in at least one matching $M_i$. It is easy to see that there exists such a collection (Vizing's Theorem on edge colouring - Theorem $5.3.2$ in \cite{Diest}).
 For each $i \in [r]$, let $G_i$ be a subgraph of $G$ such that $V(G_i) = V(G)$ and for a pair of vertices $u \in V_a$, $v \in V_b$, $\{u,v\} \in E(G_i)$ if $a=b$ or $\{h_a,h_b\} \in M_i$. Note that $G_i$ is made of $|M_i|$ disjoint components. Let $\F_i$ be a family of permutations that is pairwise suitable for $G_i$ such that $|\F_i| = \boxli(G_i)$. By Observation \ref{observationDisjointComponents}, we have $|\F_i| \leq \hat{\boxli}(P_G)$. 

From Theorem \ref{theoremBoxliSize}, $\boxlistar(H) \leq 6.84 \log r$. Let $\E$ be a family of permutations that is pairwise suitable and $3$-mixing for $H$ such that $|\E| = \boxlistar(H) \leq 6.84 \log r$. We construct two families of permutations, namely $\F_{r+1}$ and $\F_{r+2}$, of $V(G)$ from $\E$ such that $|\F_{r+1}| = |\F_{r+2}| = |\E|$. Corresponding to each permutation $\sigma \in \E$, we construct $\tau_{\sigma} \in \F_{r+1}$ and $\kappa_{\sigma} \in \F_{r+2}$ as follows. If $h_i \prec_{\sigma} h_j$, then we have $V_i \prec_{\tau_{\sigma}} V_j$ and $V_i \prec_{\kappa_{\sigma}} V_j$. Moreover, for each $i \in [r]$ and for distinct $v,v' \in V_i$, $v \prec_{\tau_{\sigma}} v'\iff v' \prec_{\kappa_{\sigma}} v$. 
\begin{claim}
\label{claimMaxPairsParts}
$\F = \bigcup_{i=1}^{r+2}\F_i$ is a pairwise-suitable family of permutations for $G$. 
\end{claim}

We prove the claim by showing that for every pair of non-adjacent edges $e, e' \in E(G)$, there is a $\sigma \in \F$ such that $e \prec_{\sigma} e'$ or $e' \prec_{\sigma} e$. We call an edge $e$ in $G$ a \emph{crossing edge} if there exists distinct $i,j \in [r]$  such that $e$ has its endpoints in $V_i$ and $V_j$. Otherwise $e$ is called a \emph{non-crossing edge}. Consider any two disjoint edges $\{a,b\}, \{c,d\}$ in $G$. Let $a \in V_i, b \in V_j, c \in V_k$ and $d \in V_l$. If $|\{i,j,k,l\}| \leq 2$, then both the edges belong to some $G_p, p \in [r]$ and hence are separated by a permutation in $\F_p$. If $|\{i,j,k,l\}| = 3$, then the two edges are separated by a permutation in $\F_{r+1}$ or $\F_{r+2}$ since $\E$ was $3$-mixing for $H$. If $|\{i,j,k,l\}| = 4$, then the two edges are separated by a permutation in both $\F_{r+1}$ and $\F_{r+2}$ since $\E$ was pairwise suitable for $H$. Details follow.

\setcounter{case}{0}
\begin{case}[both $\{a,b\}$ and $\{c,d\}$ are crossing edges]

If $i,j,k$ and $l$ are distinct then from the definition of $\E$ there exists a permutation $\sigma \in \E$ such that $\{h_i,h_j\} \prec_{\sigma} \{h_k, h_l\}$ or $\{h_k, h_l\} \prec_{\sigma} \{h_i,h_j\}$. Without loss of generality, assume $\{h_i,h_j\} \prec_{\sigma} \{h_k, h_l\}$. Therefore, in the permutations $\tau_{\sigma}$ and $\kappa_{\sigma}$ constructed from $\sigma$, we have $\{a,b\} \prec_{\tau_{\sigma}} \{c,d\}$ and $\{a,b\} \prec_{\kappa_{\sigma}} \{c,d\}$. 

Recall that $\E$ is a pairwise suitable and $3$-mixing family of permutations for $H$. If $i=k$ and $i,j,l$ are distinct, then there exists a permutation $\sigma \in \E$ such that $h_j \prec_{\sigma} h_i \prec_{\sigma} h_l$ or $h_l \prec_{\sigma} h_i \prec_{\sigma} h_j$. Without loss of generality, assume  $h_j \prec_{\sigma} h_i \prec_{\sigma} h_l$. Now it is easy to see that either $\{a,b\} \prec_{\tau_{\sigma}} \{c,d\}$ or $\{a,b\} \prec_{\kappa_{\sigma}} \{c,d\}$. The cases when $i=l, j, k$ are distinct or $i,j=k,l$ are distinct or $i,j=l,k$ are distinct are symmetric to the above case where  $i=k,j,l$ are distinct. 

Consider the case when $i=k,j=l$ are distinct. In this case, both $\{a,b\}$ and $\{c,d\}$ have their endpoints in $V_i$ and $V_j$. Then there exists some $p \in [r]$ such that $\{a,b\}, \{c,d\} \in E(G_p)$. Since $\F_p$ is a pairwise suitable family of permutations for $G_p$ there exists a $\sigma \in F_p$ such that $\{a,b\} \prec_{\sigma} \{c,d\}$ or  $\{c,d\} \prec_{\sigma} \{a,b\}$. The case when $i=l$ and $j=k$ are distinct is similar. 
\end{case}

\begin{case}[only $\{a,b\}$ is a crossing edge] 

Let $a \in V_i, b \in V_j$ and $c,d \in V_k$. If $i,j,k$ are distinct then there exists a permutation $\sigma$ in $\E$ such that either $h_i \prec_{\sigma} h_j \prec_{\sigma} h_k$ or $h_k \prec_{\sigma} h_j \prec_{\sigma} h_i$. Without loss of generality, assume $h_i \prec_{\sigma} h_j \prec_{\sigma} h_k$. Now its easy to see that both $\{a,b\} \prec_{\tau_{\sigma}} \{c,d\}$ and $\{a,b\} \prec_{\kappa_{\sigma}} \{c,d\}$. If $i=k, j$ are distinct then both $\{a,b\}$ and $\{c,d\}$ have their endpoints from $V_i \cup V_j$. Then there exists some $p \in [r]$ such that $\{a,b\}, \{c,d\} \in E(G_p)$. Since $\F_p$ is a pairwise suitable family of permutations for $G_p$ there exists a $\sigma \in \F_p$ such that $\{a,b\} \prec_{\sigma} \{c,d\}$ or  $\{c,d\} \prec_{\sigma} \{a,b\}$. The case when $j=k, i$ are distinct is similar.   
\end{case}

\begin{case}[only $\{c,d\}$ is a crossing edge] 
Similar to the case above. 
\end{case}

\begin{case}[both $\{a,b\}$ and $\{c,d\}$ are non-crossing edges]
Then, for each $p \in [r]$,  $\{a,b\}, \{c,d\} \in E(G_p)$. Since $\F_p$ is a pairwise suitable family of permutations for $G_p$ there exists a $\sigma \in \F_p$ such that $\{a,b\} \prec_{\sigma} \{c,d\}$ or  $\{c,d\} \prec_{\sigma} \{a,b\}$. 
\end{case}

Thus, we prove Claim \ref{claimMaxPairsParts}. Hence, we have $\boxli(G) \leq |\mathcal{F}| = \sum_{i=1}^{r}|F_i| + |F_{r+1}| + |F_{r+2}| \leq \hat{\boxli}(P_G) r + 13.68 \log r$.
\end{proof}


\subsection{Maximum degree}

Adiga, Bhowmick, and Chandran have shown that the boxicity of a graph $G$ of maximum degree $\Delta$ is in  $O(\Delta \log^2 \Delta)$ \cite{DiptAdiga}. For any hypergraph $H$ of rank $r$ and maximum degree $D$ the maximum degree of $L(H)$ is $r(D-1)$. Hence the next bound follows immediately form Theorem \ref{theoremConnectionBoxliPermutation}.

\begin{corollary}
\label{corollaryHypergraphMaxDegree}
For any hypergraph $H$ of rank $r$ and maximum degree $D$,
$$ \boxli(H) \in \order{rD \log^2(rD)}.
$$
\end{corollary}

It is known that there exist graphs of maximum degree $\Delta$ whose boxicity can be as high as $c \Delta \log \Delta$ \cite{DiptAdiga}, where $c$ is a small enough positive constant. Let $G$ be one such graph. Consider the following hypergraph $H$ constructed from $G$. Let $V(H) = E(G)$ and $E(H) = \{E_v : v \in V(G)\}$ where $E_v$ is the set of edges incident on the vertex $v$ in $G$. It is clear that $G = L(H)$. Hence $\boxli(H) = \boxi(G) \geq c \Delta(G) \log \Delta(G)$. Note that the rank of $H$ is $r = \Delta(G)$ and the maximum degree of $H$ is $2$. Thus $\boxli(H) \geq c r \log(r)$ and hence the dependence on $r$ in the upper bound in Corollary \ref{corollaryHypergraphMaxDegree} cannot be considerably brought down in general.

We improve the above upper bound in the case of graphs using the auxiliary lemma from the previous section. For a graph $G$ with maximum degree $\Delta$, it is easy to see that from Lemma \ref{lemmaMaxPairsParts}, $\boxli(G) \in \order{\Delta^2}$. Consider $P_G$ to be the partition of $V(G)$ corresponding to the colour classes in a distance-two colouring of $G$, i.e, a vertex colouring of $G$ in which no two vertices of $G$ which are at a distance at most $2$ from each other are given the same colour. Then the subgraphs induced by any pair of colour classes is a collection of disjoint edges and hence $\hat{\pi}(P_G) \leq 1$. It is easy to see that a distance-two colouring can be done using $\Delta^2 + 1$ colours and hence the bound. Corollary \ref{corollaryHypergraphMaxDegree} improves it to $\order{\Delta \log^2 \Delta}$. It was shown in \cite{RogSunSiv} using $3$-suitable family of permutations that $\boxli(G) \in \order{\Delta \log\log \Delta}$. Here we improve the above bound and show that $\boxli(G) \leq 2^{9 \ilog \Delta} \Delta$ (Theorem \ref{theoremBoxliDelta}). The idea employed is to recursively partition $V(G)$ into $\order{\Delta / \log \Delta}$ parts such that the subgraphs induced by any pair of parts have a maximum degree at most $\log \Delta$ and then apply Lemma \ref{lemmaMaxPairsParts}. Existence of such a partition is guaranteed by Lemma \ref{lemmaLogDegreePartition} below which in turn is proved by an application of the powerful Lov\'{a}sz local lemma.

\begin{lemma}[Lov\'{a}sz local lemma, Erd\H{o}s and Lov\'{a}sz \cite{lovaszlocallemma}] 
\label{lemmaLovaszLocal}
Let $G$ be a graph on vertex set $[n]$ with maximum degree $d$ and let $A_1, \ldots , A_n$ be events defined on some probability space such that for each $i$, 
$$Pr[A_i] \leq \frac{1}{4d}.$$ Suppose further that each $A_i$ is jointly independent of the events $A_j$ for which $\{i,j\} \notin E(G)$. Then $Pr[\overline{A_1} \cap \cdots \cap \overline{A_n}] > 0$. 
\end{lemma}

The following lemma is similar to Lemma $4.2$ in \cite{FurediKahn} and shall be used in proving an upper bound for $\boxli(G)$ in terms of the maximum degree of $G$. 

\begin{lemma}
\label{lemmaLogDegreePartition}
For a graph $G$ with maximum degree $\Delta \geq 2^{64}$, there exists a partitioning of $V(G)$  into $\ceil{ 400 \Delta / \log \Delta }$ parts such that for every vertex $v \in V(G)$ and for every part $V_i, \, i \in \big[ \ceil{ 400 \Delta / \log \Delta } \big],\, \lvert N_G(v) \cap V_i \rvert \leq \frac{1}{2} \log \Delta$.   
\end{lemma}
\begin{proof}
Since we can have a $\Delta$-regular supergraph (with possibly more vertices) of $G$ we can as well assume that $G$ is $\Delta$-regular. Let $r = \ceil{ \frac{400 \Delta}{\log \Delta} } \leq \frac{401\Delta}{\log \Delta}$. Partition $V(G)$ into $V_1 , \ldots , V_r$ using the following procedure: for each $v \in V(G)$, independently assign $v$ to a set $V_i$ uniformly at random from $V_1, \ldots , V_r$. 

We use the following well known multiplicative form of Chernoff Bound (Theorem 4.4 in \cite{mitzenmacher}). Let $X$ be a sum of mutually independent indicator random variables with $\mu = E[X]$. Then for any $\delta > 0$,
$$Pr[X \geq (1+\delta) \mu] \leq c_{\delta}^{\mu},$$
where $c_{\delta} = e^{\delta} / (1 + \delta)^{(1 + \delta)}$. 

Let $d_i(v)$ be a random variable that denotes the number of neighbours of $v$ in $V_i$. Then $\mu_{i,v} = E[d_i(v)] = \frac{\Delta}{r} \leq \frac{1}{400} \log \Delta$. For each $v \in V(G), i \in [r]$, let $E_{i,v}$ denote the event $d_i(v) \geq \frac{1}{2}\log \Delta$. Then applying the above Chernoff bound with $\delta = 199$, we have $Pr[E_{i,v}] = Pr[d_i(v) \geq 200 \frac{\log \Delta}{400}] \leq 2^{-3.1 \log \Delta} = \Delta^{-3.1}$. 
Consider the collection of ``bad'' events $E_{i,v}$, $i \in [r], v \in V(G)$. In order to apply Lemma \ref{lemmaLovaszLocal}, we construct a dependency graph $H$ whose vertices are events $E_{i,v}$ and two vertices are adjacent if and only if the corresponding two events are dependent. Since $E_{i,v}$ depends only on where the neighbours of $v$ went to in the random partitioning, it is easy to see that the maximum degree of $H$, denoted by $d_H$, is at most $(1 + \Delta + \Delta(\Delta-1))r = (1+\Delta^2)r \leq \frac{402 \Delta^3}{\log \Delta}$. For each $i \in [r], v \in V(G)$, $Pr[E_{i,v}] \leq \frac{1}{\Delta^{3.1}}  \leq \frac{\log \Delta}{1608 \Delta^3} \leq \frac{1}{4d_H}$. Therefore, by Lemma \ref{lemmaLovaszLocal}, we have $Pr[\bigcap_{i \in [r], v \in V(G)} \overline{E_{i,v}}] > 0$. Hence there exists a partition satisfying our requirements. 
\end{proof}

\begin{theorem}
For a graph $G$ with maximum degree $\Delta$, $\boxli(G) \leq 2^{9 \ilog \Delta} \Delta$.
\label{theoremBoxliDelta} 
\end{theorem}
\begin{proof}
Let $\boxli(\Delta) := \max\{\boxli(H) : H \textnormal{ is a graph with maximum degree at most } \Delta \}$. Then, clearly $\boxli(G) \leq \boxli(\Delta)$. 
If $\Delta \leq 1$, then $G$ is a collection of matching edges and disjoint vertices and therefore $\boxli(1) = 1$.  When $\Delta > 1$, it was shown in Theorem $10$ of \cite{RogSunSiv} that $\boxli(\Delta) \leq (4\Delta - 4)(\ceil{ \log \log (2\Delta - 2) } + 3) + 1$. For every $1 < \Delta < 2^{64}$, it can be verified that $(4\Delta - 4)(\ceil{ \log \log (2\Delta - 2) } + 3) + 1 \leq 2^{9 \ilog \Delta} \Delta$. Therefore, the statement of the theorem is true for every $\Delta < 2^{64}$. 

For $\Delta \geq 2^{64}$, let $P_G$ be a partition of $V(G)$ into $V_1 \uplus \cdots \uplus V_{r}$ where $r=\ceil{ 400 \Delta / \log \Delta }$ and $|N_G(v) \cap V_i| \leq \frac{1}{2} \log \Delta, ~ \forall v \in V(G), i \in [r]$. Existence of such a partition is guaranteed by lemma \ref{lemmaLogDegreePartition}. From Lemma \ref{lemmaMaxPairsParts}, we have $\boxli(G) \leq 13.68 \log r +  \hat{\boxli}(P_G) r$ where $\hat{\boxli}(P_G) = \max_{i,j \in [r]} \boxli(G[V_i \cup V_j])$. Since $|N_G(v) \cap V_i| \leq \frac{1}{2} \log \Delta$ for every $v \in V(G), i \in [r]$, the maximum degree of the graph $G[V_i \cup V_j]$ is at most $\log \Delta$ for every $i, j \in [r]$. Therefore, $\hat{\boxli}(P_G) \leq \boxli(\log \Delta)$. Thus we have 
\begin{eqnarray}
\label{equationRecurrencePi}
\boxli(\Delta) 
	& \leq & \ceil{ \frac{400 \Delta}{\log \Delta} } \boxli(\log \Delta) + 13.68\log \ceil{ \frac{400 \Delta}{\log \Delta}} \nonumber \\
	& \leq & 2^9 \frac{\Delta}{\log \Delta} \boxli(\log \Delta), \, \textnormal{where } \Delta \geq 2^{64}.
\end{eqnarray}

Now we complete the proof by using induction on $\Delta$. The statement is true for all value of $\Delta < 2^{64}$  and we have the recurrence relation of Equation (\ref{equationRecurrencePi}) for larger values of $\Delta$. For an arbitrary $\Delta \geq 2^{64}$, we assume inductively that the bound in the statement of the theorem is true for all smaller values of $\Delta$. Now since $\Delta \geq 2^{64}$, we can apply the recurrence in Equation (\ref{equationRecurrencePi}). Therefore
\begin{eqnarray*}
\label{equationInductionPi}
\boxli(\Delta) 
	& \leq & 2^9 \frac{\Delta}{\log \Delta} \boxli(\log \Delta) \\
	& \leq & 2^9 \frac{\Delta}{\log \Delta} 2^{9 \ilog(\log \Delta)} \log \Delta, \, \textnormal{(by induction)} \\
	& = & 	 2^{9 \ilog \Delta} \Delta. 
\end{eqnarray*}

\end{proof}

We believe that the bound proved above can be improved. Please see the discussion in Section \ref{sectionOpenProblems}.


\subsection{Degeneracy}
\label{sectionDegeneracy}

\begin{definition}
\label{definitionDegeneracy}
For a non-negative integer $k$, a graph $G$ is \emph{$k$-degenerate} if the vertices of $G$ can be enumerated in such a way that every vertex is succeeded by at most $k$ of its neighbours. The least number $k$ such that $G$ is $k$-degenerate is called the \emph{degeneracy} of $G$ and any such enumeration is referred to as a \emph{degeneracy order} of $V(G)$. 
\end{definition}

For example, trees and forests are 1-degenerate and planar graphs are 5-degenerate. Series-parallel graphs, outerplanar graphs, non-regular cubic graphs, circle graphs of girth at least 5 etc. are 2-degenerate. 

For any non-negative integer $n$, a \emph{star} $S_n$ is a rooted tree on $n+1$ nodes with one root and $n$ leaves connected to the root. In other words, a star is a tree with at most one vertex whose degree is not one. A \emph{star forest} is a  disjoint union of stars. 

\begin{definition}
\label{definitionArboricty}
The \emph{arboricity} of a graph $G$, denoted by $\mathcal{A}(G)$, is the minimum number of spanning forests whose union  covers all the edges of $G$. The \emph{star arboricity} of a graph $G$, denoted by $\mathcal{S}(G)$, is the minimum number of spanning star forests whose union covers all the edges of $G$. 
\end{definition}

Clearly, $\mathcal{S}(G) \geq \mathcal{A}(G)$ from definition. Furthermore, since any tree can be covered by two star forests, $\mathcal{S}(G) \leq 2\mathcal{A}(G)$. 

For the sake of completeness, we give a proof for the following already-known lemma. 
\begin{lemma}
\label{lemmaStarArboricityDegeneracy}
For a $k$-degenerate graph $G$, $\mathcal{S}(G) \leq 2k$.  
\end{lemma}
\begin{proof}
By following the degeneracy order, the edges of $G$ can be oriented acyclically such that each vertex has an out-degree at most $k$. Now the edges of $G$ can be partitioned into $k$ spanning forests by choosing a different forest for each outgoing edge from a vertex. Thus, $\mathcal{A}(G) \leq k$ and $\mathcal{S}(G) \leq 2k$.   
\end{proof}

\begin{theorem}
\label{theoremBoxliDegeneracy}
For a $k$-degenerate graph $G$ on $n$ vertices, $\boxli(G) \in O(k \log \log n)$.  
\end{theorem}
\begin{proof}
Let $B = \{b_1, \ldots , b_n\}$ and let $r = \floor{ \log \log n + \frac{1}{2}\log \log \log n + \log(\sqrt{2} \pi) + o(1) }$.  From \cite{scramble}, we know that there exists a family $\E =\{\sigma^1, \ldots , \sigma^r\}$ of permutations of $B$ that is $3$-suitable for $B$. Recall that a family $\E$ of permutations of $[n]$ is called $3$-suitable if for every $a, b_1, b_2 \in [n]$ their exists a permutation $\sigma \in \E$ such that $\{b_1, b_2\} \prec_{\sigma} \{a\}$.

By Lemma \ref{lemmaStarArboricityDegeneracy}, we can partition the edges of $G$ into a collection of $2k$ spanning star forests. Let $\mathcal{C} = \{C_1, \ldots , C_{2k} \}$ be one such 
collection. Each star in each star forest has exactly one root vertex which is a highest degree vertex in the star (ties resolved arbitrarily). 

Consider a spanning forest $C_i$, $i \in [2k]$. We construct a family $\F_i = \{\sigma_i^1 , \ldots, \sigma_i^r , \overline{\sigma}_i^1 , \ldots, \overline{\sigma}_i^r\}$ of permutations of $V(G)$ from $C_i$ as follows. In the permutation $\sigma_i^j$, the vertices of the same star of $C_i$ come together as a block, the blocks are ordered according to the permutation $\sigma^j$; within every block the root vertex comes last; and the leaves are ordered according to $\sigma^j$. The permutation $\overline\sigma_i^j$ is similar to $\sigma_i^j$ except that the blocks are ordered in the reverse order. This is formalised in Construction \ref{constructionBoxliDegeneracy}. Let $L_i$ and $l_i$, $i \in [2k]$ be functions from $V(G) \into B$ such that the following two properties hold.

\setcounter{property}{0}
\begin{property}
\label{property1Degeneracy}
$L_i(u) = L_i(v)$ if and only if $u$ and $v$ belong to the same star in $C_i$ 
\end{property}
\begin{property}
\label{property2Degeneracy}
If $u$ and $v$ belong to the same star in $C_i$, then $l_i(u) \neq l_i(v)$. 
\end{property}

It is straight forward to construct such functions.

\begin{construction}(Constructing $\sigma_i^j$ and $\overline{\sigma}_i^j$).
\label{constructionBoxliDegeneracy}
\begin{algorithmic}
\vspace{1ex}
\STATE{For any distinct $u,v \in V(G)$, }
\IF{$L_i(u) \neq L_i(v)$} 
\STATE{/*$u$ and $v$ belong to different stars in $C_i$ */}
\STATE{$u \prec_{\sigma_i^j} v \iff L_i(u) \prec_{\sigma^j} L_i(v)$}
\STATE{$u \prec_{\overline{\sigma}_i^j} v \iff L_i(v) \prec_{\sigma^j}  L_i(u)$}
\ELSE 
\STATE{/*$u$ and $v$ belong to the same star in $C_i$ */}
\IF{$u$ is the root vertex of its star in $C_i$}
\STATE{$v \prec_{\sigma_i^j} u$}
\STATE{$v \prec_{\overline{\sigma}_i^j} u$}
\ELSIF{$v$ is the root vertex of its star in $C_i$}
\STATE{$u \prec_{\sigma_i^j} v$}
\STATE{$u \prec_{\overline{\sigma}_i^j} v$}
\ELSE
\STATE{$u \prec_{\sigma_i^j} v \iff l_i(u) \prec_{\sigma^j} l_i(v)$}
\STATE{$u \prec_{\overline{\sigma}_i^j} v \iff l_i(u) \prec_{\sigma^j} l_i(v)$}
\ENDIF
\ENDIF
\end{algorithmic}
\end{construction} 

\begin{claim}
\label{claimBoxliDegeneracy}
$\mathcal{F} =  \bigcup_{i=1}^{2k} \F_i$ is a pairwise-suitable family of permutations for $G$. 
\end{claim}

Let $\{a,b\}, \{c,d\}$ be two disjoint edges in $G$. Let $C_i$ be the star forest which contains the edge $\{a, b \}$. We will show that one of the permutations in $\F_i$ constructed above will separate these two edges. Since the edge $\{a, b\}$ is present in $C_i$ for some $i \in [2k]$, the vertices $a$ and $b$ belong to the same star, say $S$, of $C_i$ with one of them, say $a$, as the root of $S$. If the vertices $c$ and $d$ are not in $S$ then $3$-suitability among the stars (blocks) is sufficient to separate the two edges. If $c$ and $d$ are in $S$, then the $3$-suitability within the leaves of $S$ suffices. If only one of $c$ or $d$ is in $S$, then the $3$-suitability among the leaves is sufficient to realise the separation of the two edges in one of the two corresponding permutations of the blocks. The details follow.

\setcounter{case}{0}

%
%

\begin{case}[$c,d \in V(S)$]
Then by Property \ref{property1Degeneracy}, $L_i(a) = L_i(b) = L_i(c) = L_i(d)$. Since $\E = \{\sigma_1 , \ldots , \sigma_r\}$ is a $3$-suitable family of permutations for $B= \{b_1, \ldots , b_n \}$, there exists a permutation, say $\sigma^j \in \E$, such that $\{l_i(c), l_i(d)\} \prec_{\sigma^j} \{l_i(b)\}$. Then, from Construction \ref{constructionBoxliDegeneracy}, we have $\{c,d\} \prec_{\sigma_i^j} b$. Since $a$ is the root vertex of the star $S$ in $C_i$ we also have $u \prec_{\sigma_i^j} a$, for all $u \in V(S) \setminus \{a\}$. Thus, $\{c,d\} \prec_{\sigma_i^j} \{a,b\}$.
\end{case}

\begin{case}[only $c \in V(S)$] 
Then, by Property \ref{property1Degeneracy}, $L_i(a) = L_i(b) = L_i(c)$ and $L_i(c) \neq L_i(d)$. Moreover, by Property \ref{property2Degeneracy}, $l_i(a)$, $l_i(b)$ and $l_i(c)$ are distinct. Since $\E$ is a $3$-suitable family of permutations for $B$, there exists a $\sigma^j \in \E$ such that $l_i(c) \prec_{\sigma^j} l_i(b)$. Combining this with the fact that $a$ is the root vertex of $S$, using Construction \ref{constructionBoxliDegeneracy}, we get $c \prec_{\sigma_i^j} b \prec_{\sigma_i^j} a$ and $c \prec_{\overline{\sigma}_i^j} b \prec_{\overline{\sigma}_i^j} a$. Recall that $L_i(c) \neq L_i(d)$. If $L_i(d) < L_i(c)$, then we get $d \prec_{\sigma_i^j} c \prec_{\sigma_i^j} b \prec_{\sigma_i^j} a$. Otherwise, we get $d \prec_{\overline{\sigma}_i^j} c \prec_{\overline{\sigma}_i^j} b \prec_{\overline{\sigma}_i^j} a$. 
\end{case}

\begin{case}[only $d \in V(S)$] 
This is similar to the previous subcase. 
\end{case}

\begin{case}[$c,d \notin V(S)$] 
If $c$ and $d$ belong to the same star in $C_i$, say $S'$, then by Property $P_1$,  we have $L_i(a) = L_i(b)$, $L_i(c) = L_i(d)$, and $L_i(a) \neq L_i(c)$. Then for any $j \in [r]$, either $L_i(a)  \prec_{\sigma^j} L_i(c)$ or $L_i(c) \prec_{\sigma^j} L_i(a)$. Therefore, either $\{a,b\} \prec_{\sigma_i^j} \{c,d\}$ or $\{c,d\} \prec_{\sigma_i^j} \{a,b\}$. If $c$ and $d$ belong to different stars in $C_i$, then Property $P_1$ ensures that $L_i(c)$, $L_i(d)$ and $L_i(a)$ are distinct. Since $\E$ is a $3$-suitable family of permutations for $B$, there exists a $\sigma^j \in \E$ such that $\{L_i(c), L_i(d)\} \prec_{\sigma^j} L_i(a)$. This, combined with Construction \ref{constructionBoxliDegeneracy}, implies that $\{c,d\} \prec_{\sigma_i^j} \{a,b\}$. 
\end{case}

Thus, we prove Claim \ref{claimBoxliDegeneracy}. Applying the same, we get $\boxli(G) \leq |\F| = \sum_{i=1}^{2k}|\F_i| = 4kr = 4k\floor{ \log \log n + \frac{1}{2}\log \log \log n + \log(\sqrt{2} \boxli) + o(1)  }$. 
\end{proof}

\subsubsection*{Tightness of Theorem \ref{theoremBoxliDegeneracy}} 

Let $K_n^{1/2}$ denote the graph obtained by subdividing every edge of a complete graph on $n$ vertices. Note that $K_n^{1/2}$ is $2$-degenerate. In Theorem \ref{theoremKnHalf} of {Section \ref{sectionSubdividedClique}, it is shown that $\boxli(K_n^{1/2}) \in \theta(\log \log n)$. Hence the $\log \log n$ factor in  Theorem \ref{theoremBoxliDegeneracy} cannot be brought down in general. 


\subsection{Treewidth}
\label{sectionTreewidth}

\begin{definition}
\label{defintionTreewdith}
A \emph{tree decomposition} of a graph $G$ is a pair $(\{X_i : i \in I\}, T)$, where $I$ is an index set, $\{X_i : i \in I\}$ is a collection of subsets of $V(G)$, and $T$ is a tree on $I$ such that 
\begin{enumerate}[(i)]
\item $\bigcup_{i \in I} X_i = V(G)$,
\item $\forall \{u, v\} \in E(G), \exists i \in I$ such that $u, v \in X_i$, and
\item $\forall i, j, k \in I$: if $j$ is on the path in $T$ from $i$ to $k$, then $X_i \cap X_k \subseteq X_j$.
\end{enumerate}
The {\em width} of a tree decomposition $(\{X_i~:~i \in I\}, T)$ is $\max_{i \in I} |X_i| -1$. The \emph{treewidth} of $G$ is the minimum width over all tree decompositions of $G$ and is denoted by $\tw(G)$.
\end{definition}

\begin{definition}
\label{definitionOrderedTreeDecomposition}
A tree decomposition $(\{X_i\}_{i \in V(T)}, T)$ of a graph $G$, such that $T$ has a designated root, denoted by $root(T)$, and a fixed ordering on the children of every node is called an {\em ordered tree decomposition}. By $preorder(i)$ and $postorder(i)$ we denote, respectively, the first and last time that a node $i \in V(T)$ is visited by a depth first traversal of $T$ starting from $root(T)$. For every node $i \in V(T)$, the distance from $root(T)$ in $T$ is called its {\em level} and denoted by $level(i)$. 
For a vertex $v \in V(G)$, $bag(v)$ denotes the node $i \in V(T)$ at the smallest level such that $v \in X_i$. Finally, $T(v)$ denotes the subtree of $T$ induced by $bag(v)$ and all its descendents.
\end{definition}

It follows from the above definition that for every $u, v \in V(G)$ either $T(u)$ and $T(v)$ are disjoint or one is contained in the other depending on whether one is a descendent of the other or not. Hence the following observation is immediate. We use $T(u) \subseteq T(v)$ to denote that $T(u)$ is contained in $T(v)$. 

\begin{observation}
\label{observationSubtrees}
Let $(\{X_i\}_{i \in V(T)}, T)$ be an ordered tree decomposition of a graph $G$. For every $\{u, v\} \in E(G)$, either $T(u) \subseteq T(v)$ or $T(v) \subseteq T(u)$.
\end{observation}

\begin{definition}
\label{definitionPSplittingPreorder}
Let $\mathcal{T} = (\{X_i\}_{i \in V(T)}, T)$ be an ordered tree decomposition of a graph $G$ and let $P = (V_1, V_2)$ be a bipartition of V(G), i.e., $V_1 \uplus V_2 = V(G)$. We define a function $f: V(G) \into \N$ as follows. 
\[
f(v) = 
	\begin{cases}
	\begin{array}{ll}
	preorder(bag(v)), & \textnormal{if } v \in V_1 \\
	postorder(bag(v)), & \textnormal{if } v \in V_2
	\end{array}
	\end{cases}
\] 
A permutation $\sigma$ of $V(G)$ is called {\em $P$-splitting} if $f(u) < f(v) \implies u \prec_{\sigma} v$.
\end{definition}

\begin{theorem}
\label{theoremBoxliTreewidth}
Let $G$ be a graph of treewidth $t$. Then $\boxli(G) \leq 15.68 \ceil{\log(t+1)} + 2$.
\end{theorem}
\begin{proof}
Let $\mathcal{T} = (\{X_i\}_{i \in V(T)}, T)$ be an ordered tree decomposition of $G$ of width $t$. Let $G'$ be a supergraph of $G$ obtained by adding an edge between every pair of vertices that appear together in some bag $X_i, i \in V(T)$. Hence the treewidth of $G'$ is also $t$ and so its chromatic number is $t+1$. Let $c : V(G') \into [t+1]$ be a proper colouring of $G'$. In the proof to follow, we shall prove the theorem for $G'$. Since $G'$ is a supergraph of $G$, by Observation \ref{observationMonotonicity}, the theorem follows. 

Let $K_{t+1}$ be a complete graph on $[t+1]$ and let $\E$ be a smallest family of permutations that is pairwise suitable and $3$-mixing for $K_{t+1}$. By Theorem \ref{theoremBoxliSize}, we know that $|\E| \leq 6.84 \log(t+1)$. For $\sigma \in \E$, let $(V(G), \lhd_{\sigma})$ be the partial order in which $u \lhd_{\sigma} v \iff c(u) \prec_{\sigma} c(v)$. Let $\tau(\sigma)$ and $\tau'(\sigma)$ be two linear extensions of $(V(G), \lhd_{\sigma})$ such that for two distinct vertices $u, v \in V(G)$ with $c(u) = c(v)$, we have $u \prec_{\tau(\sigma)} v \iff v \prec_{\tau'(\sigma)} u$. Let $\F_1 = \{\tau(\sigma), \tau'(\sigma)\}_{\sigma \in \E}$.

Consider two disjoint edges $\{u_1, u_2\}$ and $\{u_3, u_4\}$ of $G'$. Let $C = \{c(u_i): i \in [4]\}$. If $|C| = 4$, that is if all the four end points have different colours, then consider the permutation $\sigma \in \E$ that separates $\{c(u_1), c(u_2)\}$ from $\{c(u_3), c(u_4)\}$. It is easy to see that $\{u_1, u_2\}$ is separated from $\{u_3, u_4\}$ in both $\tau(\sigma)$ and $\tau'(\sigma)$. If $|C| = 3$, then without loss of generality, we can assume that $c(u_1) = c(u_3)$. Since $\E$ is $3$-mixing for $K_{t+1}$, there exists a permutation $\sigma \in \E$ such that $c(u_1)$ is between $c(u_2)$ and $c(u_4)$ in $\sigma$. Hence $\{u_1, u_2\}$ and $\{u_3, u_4\}$ are separated in exactly one of $\tau(\sigma)$ or $\tau'(\sigma)$.

The case left to be considered is the case when $|C| = 2$. In this case, we construct a different family of permutations. Let $\mathcal{P}$ be a family of bipartitions of $V(G)$ such that for every pair of distinct colours $i, j \in [t+1]$, there exists a partition $(V_1, V_2) \in \mathcal{P}$ with $c^{-1}(i) \subseteq V_1$ and $c^{-1}(j) \subseteq V_2$. It is easy to see that we can have such a family of size $2 \ceil{ \log(t+1)}$ by partitioning $V(G)$ based on the bits of a binary encoding of colours. For a bipartition $P$ of $V(G)$, let $\sigma(P)$ denote the $P$-splitting permutation of $V(G)$ as in Definition \ref{definitionPSplittingPreorder}. In particular, $\sigma_{pre} = \sigma((V(G), \emptyset))$ and $\sigma_{post} = \sigma((\emptyset, V(G)))$. Finally, let $\F_2 = \{\sigma((V_1, V_2)) : (V_1, V_2) \in \mathcal{P} \} \cup \{\sigma_{pre}, \sigma_{post} \}$.

Since $|C| = 2$ we can assume without loss of generality that $c(u_1) = c(u_3) = i$ and $c(u_2) = c(u_4) = j$, $i \neq j$. Let $(V_1, V_2) \in \mathcal{P}$ be the bipartition such that $c^{-1}(i) \subseteq V_1$ and $c^{-1}(j) \subseteq V_2$. Similarly let $(U_1, U_2) \in \mathcal{P}$ be the partition such that $c^{-1}(i) \subseteq U_2$ and $c^{-1}(j) \subseteq U_1$. Let $\sigma_{ij} = \sigma((V_1, V_2))$ and $\sigma_{ji} =  \sigma((U_1, U_2))$. We claim that one of the permutations from $\{\sigma_{ij}, \sigma_{ji}, \sigma_{pre}, \sigma_{post} \}$ will separate $\{u_1, u_2\}$ from $\{u_3, u_4\}$.

Without loss of generality we can assume that $level(bag(u_1)) \leq level(bag(u_i)),~ \forall i \in [4]$. So $T(u_2) \subseteq T(u_1)$ (Observation \ref{observationSubtrees}). If $T(u_3) \cup T(u_4)$ is disjoint from $T(u_1)$ then $\sigma_{pre}$ separates $\{u_1, u_2\}$ from $\{u_3, u_4\}$. So we can assume $T(u_3) \cup T(u_4) \subseteq T(u_1)$. If $preorder(bag(u_2)) < preorder(bag(u_i)),~\forall i \in \{3, 4 \}$, then $\sigma_{pre}$ will separate  $\{u_1, u_2\}$ from $\{u_3, u_4\}$. Similarly if  $postorder(bag(u_2)) > postorder(bag(u_i)),~\forall i \in \{3,4\}$, then $\sigma_{post}$ will separate them. Hence we can further assume that $T(u_2) \subseteq T(u_3) \cup T(u_4)$. Since $u_1$ and $u_2$ are adjacent and $c(u_1) = c(u_3)$, it can be seen that once $T(u_3) \subseteq T(u_1)$ as we have here, we cannot have $T(u_2) \subseteq T(u_3)$. Since $u_3$ and $u_4$ are adjacent, we get $T(u_2), T(u_3) \subsetneq T(u_4) \subseteq T(u_1)$ with $T(u_2) \not\subseteq T(u_3)$. Since $c(u_2) = c(u_4)$, by a similar argument, $T(u_3) \not\subseteq T(u_2)$. Hence we can conclude that $T(u_2) \cap T(u_3) = \emptyset$.  Now if $postorder(bag(u_2)) < preorder(bag(u_3))$, then $\sigma_{ij}$ separates $\{u_1, u_2\}$ from $\{u_3, u_4\}$. Otherwise, $postorder(bag(u_3)) < preorder(bag(u_2))$ and therfore $\sigma_{ji}$ does the required separation.

Hence we conclude that $\F_1 \cup \F_2$ is a pairwise suitable family of permutations for $G'$ and hence $G$. Therefore $\boxli(G) \leq 15.68 \ceil{\log(t+1)} + 2$.
\end{proof}

\subsubsection*{Tightness of Theorem \ref{theoremBoxliTreewidth}} For a complete graph $K_n$, $\tw(K_n) = n-1$. By Corollary \ref{corollaryBoxliOmega} in Section \ref{sectionLowerBound}, we have $\boxli(K_n) \geq \log \floor {n/2}$. Hence, Theorem \ref{theoremBoxliTreewidth} is tight up to a constant factor.


\subsection{Acyclic and star chromatic number}

\begin{definition}
\label{definitionAcyclicStarChromatic}
The \emph{acyclic chromatic number} of a graph $G$, denoted by $\chi_a(G)$, is the minimum number of colours needed to do a proper colouring of  the vertices of $G$ such that the graph induced on the vertices of every pair of colour classes is acyclic. The \emph{star chromatic number} of a graph $G$, denoted by $\chi_s(G)$, is the minimum number of colours needed to do a proper colouring of  the vertices of $G$ such that the graph induced on the vertices of every pair of colour classes is a star forest. 
\end{definition}

Recall (from Section \ref{sectionDegeneracy}) that a star forest is a disjoint union of stars. Clearly, $\chi_s(G) \geq \chi_a(G) \geq \chi(G)$, where $\chi(G)$ denotes the chromatic number of $G$. In order to bound $\boxli(G)$ in terms of $\chi_a(G)$ and $\chi_s(G)$, we first bound $\boxli(G)$ for forests and star forests. Then the required result follows from an application of Lemma \ref{lemmaMaxPairsParts} from Section \ref{sectionBoxliSize}.

\begin{lemma}
\label{lemmaBoxliStars}
For a star forest $G$, $\boxli(G) = 1$.  
\end{lemma}
\begin{proof}
Let $S_1, \ldots , S_r$ be the collection of stars that form $G$. Let $\sigma$ be a permutation of $V(G)$ which satisfies $V(S_1) \prec_{\sigma} \cdots \prec_{\sigma} V(S_r)$. It is easy to verify that $\{\sigma\}$ is pairwise suitable for $G$. 
\end{proof}

\begin{lemma}
\label{lemmaBoxliForests}
For a forest $G$, $\boxli(G) \leq 2$.  
\end{lemma}
\begin{proof}
Let $T_1, \ldots , T_r$ be the collection of trees that form $G$.  Convert each tree $T_i$ to an ordered tree by arbitrarily choosing a root vertex for $T_i$ and assigning an arbitrary order to the children of each vertex. Let $\sigma_1, \sigma_2$ be two permutations of $V(G)$ defined as explained below. Consider a vertex $u \in V(T_i)$ and a vertex $v \in V(T_j)$, where $i,j \in [r]$. If $i \neq j$, then $u \prec_{\sigma_1} v \iff i < j$ and $u \prec_{\sigma_2} v \iff i < j$. Otherwise, $u \prec_{\sigma_1} v$ if and only if $u$ precedes $v$ in a preorder traversal of the ordered tree $T_i$ and $u \prec_{\sigma_2} v$ if and only if $u$ precedes $v$ in a postorder traversal of the ordered tree $T_i$. It is left to the reader to verify that $\{\sigma_1, \sigma_2\}$ form  pairwise suitable family of permutations for $G$. 
\end{proof}

\begin{theorem}
\label{theoremBoxliAcyclicStar}
For a graph $G$, $\boxli(G) \leq 2\chi_a(G) + 13.68\log(\chi_a(G))$. Further, if the star chromatic number of $G$ is $\chi_s$, then   $\boxli(G) \leq \chi_s(G) + 13.68\log(\chi_s(G))$. 
\end{theorem}
\begin{proof}
The theorem follows directly from Lemma \ref{lemmaMaxPairsParts}, Lemma \ref{lemmaBoxliForests}, and Lemma \ref{lemmaBoxliStars}.  
\end{proof}

This, together with some existing results from literature, gives us a few easy corollaries. Alon, Mohar, and Sanders have showed that a graph embeddable in a surface of Euler genus $g$ has an acyclic chromatic number in $O(g^{4/7})$ \cite{alonacyclicgenus}. It is noted by Esperet and Joret in \cite{esperet2011boxicity}, using results of Nesetril, Ossona de Mendez, Kostochka, and Thomason, that graphs with no $K_t$ minor have an acyclic chromatic number in $\order{t^2 \log t}$. Hence the following corollary.

\begin{corollary}
\label{corollaryBoxliGenus}
\begin{enumerate}[(i)]
\item For a graph $G$ with Euler genus $g$, $\boxli(G) \in O(g^{4/7})$, 
\item for a graph $G$ with no $K_t$ minor, $\boxli(G) \in O(t^2 \log t)$, and 
\end{enumerate}
\end{corollary}


\subsection{Planar graphs}

Since planar graphs have acyclic chromatic number at most $5$ \cite{borodin1979acyclic}, it follows from Theorem \ref{theoremBoxliAcyclicStar} that, for every planar graph $G$, $\boxli(G) \leq 42$. Using Schnyder's celebrated result on non-crossing straight line plane drawings of planar graphs we improve this bound to the best possible.

\begin{theorem}[Schnyder, Theorem $1.1$ in \cite{schnyder1990embedding}]
\label{theoremSchnyder}
Let $\lambda_1$, $\lambda_2$, $\lambda_3$ be three pairwise non parallel straight lines in the plane. Then, each plane graph has a straight line embedding in which any two disjoint edges are separated by a straight line parallel to $\lambda_1$, $\lambda_2$ or $\lambda_3$. 
\end{theorem} 

This immediately gives us the following tight bound for planar graphs.

\begin{theorem}
\label{theoremBoxliPlanar}
Separation dimension of a planar graph is at most $3$. More over there exist planar graphs with separation dimension $3$.
\end{theorem}
\begin{proof}
Consider the following three pairwise non parallel lines in $\R^2$: $\lambda_1 = \{(x,y) : y = 0, x \in \R \} $, $\lambda_2 = \{(x,y) : x = 0, y \in \R \}$ and $\lambda_3 = \{(x,y): x,y \in \R, x+y = 0\}$. Let $f: V(G) \into \R^2$ be  an embedding  such that any two disjoint edges in $G$ are separated by a straight line parallel to $\lambda_1$, $\lambda_2$ or $\lambda_3$. For every vertex $v$, let $v_x$ and $v_y$ denote the projections of $f(v)$ on to the $x$ and $y$ axes respectively. 

Construct $3$ permutations $\sigma_1, \sigma_2, \sigma_3$ such that $u_x < v_x \implies u \prec_{\sigma_1} v$, $u_y < v_y \implies u \prec_{\sigma_2} v$, and $u_x + u_y < v_x + v_y \implies u \prec_{\sigma_3} v$, with ties broken arbitrarily. Now it is easy to verify that any two disjoint edges of $G$ separated by a straight line parallel  to $\lambda_i$ in the embedding $f$, will be separated in $\sigma_i$.

Tightness of the theorem follows from considering $K_4$, the complete graph on $4$ vertices which is a planar graph. Any single permutation of its $4$ vertices separates exactly one pair of disjoint edges. Since $K_4$ has $3$ pairs of disjoint edges, we need exactly $3$ permutations.
\end{proof}


\subsection{Subdivisions of graphs} 

\begin{definition}
A graph $G'$ is called a {\em subdivision} of a graph $G$ if $G'$ is obtained from $G$ by replacing a subset of edges of $G$ with independent paths between their ends such that none of these new paths has an inner vertex on another path or in $G$. A subdivision of $G$ where every edge of $G$ is replaced by a $k$-length path is denoted as $G^{1/k}$. The graph $G^{1/2}$ is called {\em fully subdivided} $G$.
\end{definition}

The main result in this section is an upper bound for $\boxli(G^{1/2})$ in terms of $\chi(G)$, where $\chi(G)$ denotes the chromatic number of $G$. It is easy to see that the acyclic chromatic number of $G^{1/k}$ for $k \geq 3$ is at most $3$ for any graph $G$ (Use the first two colours to properly colour the internal vertices in every path introduced by the subdivision and give the third colour to all the original vertices) \cite{wood2005acyclic}. Hence, by Theorem \ref{theoremBoxliAcyclicStar}, $\boxli(G^{1/k}) \in \order{1}, \forall k > 2$.  
 Acyclic chromatic number of $G^{1/2}$ is at most $\max \{ \chi(G), 3 \}$ \cite{wood2005acyclic} and hence $\boxli(G^{1/2}) \in \order{\chi(G)}$ by Theorem \ref{theoremBoxliAcyclicStar}. We improve this easy upper bound considerably and show that $\boxli(G^{1/2}) \leq (1 + o(1))\log\log \chi(G)$. In Section \ref{sectionSubdividedClique}, we come up with a different strategy to show that $\boxli(K_n^{1/2}) \geq \frac{1}{2} \floor{\log \log (n-1)}$ there by demonstrating the tightness of the above upper bound.

The upper bound on $\boxli(G^{1/2})$ is obtained by a constructing an interval order based on $G$ of height $\chi(G) - 1$ and then showing that its poset dimension is an upper bound on $\boxli(G^{1/2})$. We need some more definitions and notation before proceeding. 

\begin{definition}[Poset dimension]
Let $(\mathcal{P}, \lhd)$ be a poset (partially ordered set). A {\em linear extension} $L$ of $\mathcal{P}$ is a total order which satisfies $(x \lhd y \in \mathcal{P}) \implies (x \lhd y \in L)$. A {\em realiser} of $\mathcal{P}$ is a set of linear extensions of $\mathcal{P}$, say $\mathcal{R}$, which satisfy the following condition: for any two distinct elements $x$ and $y$, $x\lhd y \in \mathcal{P}$ if and only if $x \lhd y \in L$, $\forall L \in \mathcal{R}$.  
The \emph{poset dimension} of $\mathcal{P}$, denoted by $dim(\mathcal{P})$, is the minimum integer $k$ such that there exists a realiser of $\mathcal{P}$ of cardinality $k$. 
\end{definition}

\begin{definition}[Interval dimension]
A \emph{open interval} on the real line, denoted as $(a,b)$, where $a,b \in \R$ and $a < b$, is the set $\{x \in \mathbb{R} : a <  x < b\}$. For a collection $C$ of open intervals on the real line the partial order $(C, \lhd)$ defined by the relation $(a,b) \lhd (c,d)$ if $b \leq c$ in $\R$ is called the {\em interval order} corresponding to $C$. The poset dimension of this interval order $(C,\lhd)$ is called the \emph{interval dimension} of $C$ and is denoted by $\idim(C)$.
\end{definition}

\begin{theorem}
\label{theoremSubdivisionIntervalOrder}
For any graph $G$ and a permutation $\sigma$ of $V(G)$, let $C_{G, \sigma}$ denote the collection of open intervals $(\sigma(u), \sigma(v)), \{u,v\} \in E(G), u \prec_{\sigma} v$. Then,
$$ \boxli(G^{1/2}) \leq \min_{\sigma} \idim ( C_{G,\sigma} ) + 2,$$
where the minimisation is done over all possible permutations $\sigma$ of $V(G)$.
\end{theorem}
\begin{proof}
Let $\sigma$ be any permutation of $V(G)$. We relabel the vertices of $G$ so that $v_1 \prec_{\sigma} \cdots \prec_{\sigma} v_n$, where $n = |V(G)|$. For every edge $e = \{v_i, v_j\} \in E(G), i < j$, the new vertex in $G^{1/2}$ introduced by subdividing $e$ is denoted as $u_{ij}$. For a new vertex $u_{ij}$, its two neighbours, $v_i$ and $v_j$ will be respectively called the {\em left neighbour} and {\em right neighbour} of $u_{ij}$. We call an edge of the form $\{v_i, u_{ij}\}$ as a {\em left edge} and one of the form $\{u_{ij}, v_j\}$ as a {\em right edge}.

Let $\mathcal{R} = \{L_1, \ldots, L_d\}$ be a realiser for $(C_{G, \sigma}, \lhd)$ such that $d = \idim(C_{G,\sigma})$. For each total order $L_p, p \in [d]$, we construct a permutation $\sigma_p$ of $V(G^{1/2})$ as follows. First, the subdivided vertices are ordered from left to right as the corresponding intervals are ordered in $L_p$, i.e, $u_{ij} \prec_{\sigma_p} u_{kl} \iff (i,j) \prec_{L_p} (k,l)$. Next the original vertices are introduced into the order one by one as follows. The vertex $v_1$ is placed as the left most vertex. Once all the vertices $v_i, i < j$ are placed, we place $v_j$ at the left most possible position so that $v_{j-1} \prec_{\sigma_p} v_j$ and $u_{ij} \prec_{\sigma_p} v_j, \forall i <j$. This ensures that $v_j \prec_{\sigma_p} u_{jk}, \forall k >j$ because $u_{ij'} \prec_{\sigma_p} u_{jk}, \forall j' \leq j$ (Since $(i,j) \lhd (j,k)$). Now we construct two more permutations $\sigma_{d+1}$ and $\sigma_{d+2}$ as follows. In both of them, first the original vertices are ordered as $v_1 \prec \cdots \prec v_n$. In $\sigma_{d+1}$, the subdivided vertices are placed immediately after its left neighbour, i.e., $v_i \prec_{\sigma_{d+1}} u_{ij} \prec_{\sigma_{d+1}} v_{i+1}$ for all $\{i, j \} \in E(G)$. In $\sigma_{d+2}$, the subdivided vertices are placed immediately before its right neighbour, i.e., $v_{j-1} \prec_{\sigma_{d+2}} u_{ij} \prec_{\sigma_{d+2}} v_{j}$ for all $\{i, j \} \in E(G)$. Notice that in all the permutations so far constructed, the left (right) neighbour of every subdivided vertex is placed to its left (right).

We complete the proof by showing that $\F = \{\sigma_1, \ldots, \sigma_{d+2}\}$ is pairwise suitable for $G^{1/2}$ by analysing the following cases.  Any two disjoint left edges are separated in $\sigma_{d+1}$ and any two disjoint right edges are separated in $\sigma_{d+2}$. If $(i,j) \lhd (k,l)$, then every pair of disjoint edges among those incident on $u_{ij}$ or $u_{kl}$ are separated in every permutation in $\F$. Hence the only non-trivial case is when we have a left edge $\{v_i, u_{ij}\}$ and a right edge $\{u_{kl}, v_l\}$ such that $(i,j) \cap (k,l) \neq \emptyset$. Since $(i,j)$ and $(k,l)$ are incomparable in $(C_{G, \sigma}, \lhd)$, there exists a permutation $\sigma_p, p \in [d]$ such that $u_{ij} \prec_{\sigma_p} u_{kl}$. Since $v_i$ is before $u_{ij}$ and $v_l$ is after $u_{kl}$ in every permutation, $\sigma_p$ separates $\{v_i, u_{ij}\}$ from $\{u_{kl}, v_l\}$.  
\end{proof}

The {\em height} of a partial order is the size of a largest chain in it. It was shown by F\"{u}redi, Hajnal, R\"{o}dl and Trotter \cite{furedi1991interval} that the dimension of an interval order of height $h$ is at most $\log\log h + (\frac{1}{2} + o(1))\log\log\log h$ (see also Theorem $9.6$ in \cite{trotter1997new}). The next corollary uses this result along with Theorem \ref{theoremSubdivisionIntervalOrder}.

\begin{corollary}
\label{corollarySubdivisionChromaticNumber}
For a graph $G$ with chromatic number $\chi(G)$, 
$$ 
	\boxli(G^{1/2}) \leq 
		\log\log (\chi(G)-1) + 
		\left( \frac{1}{2} + o(1) \right) \log\log\log (\chi(G)-1) + 2.
$$
\end{corollary}
\begin{proof}
Let $V_1, \ldots, V_{\chi(G)}$ be the colour classes of an optimal proper colouring of $G$. Let $\sigma$ be a permutation of $V(G)$ such that $V_1 \prec_{\sigma} \cdots \prec_{\sigma} V_{\chi(G)}$. Now it is easy to see that the longest chain in $(C_{G,\sigma}, \lhd)$ is of length at most $\chi(G) - 1$. Hence the result follows from that of F\"{u}redi et al. \cite{furedi1991interval} and Theorem \ref{theoremSubdivisionIntervalOrder} above.
\end{proof}

\subsubsection*{Tightness of Corollary \ref{corollarySubdivisionChromaticNumber}}
Theorem \ref{theoremKnHalf} in Section \ref{sectionSubdividedClique} proves that $\boxli(K_n^{1/2}) \geq \frac{1}{2} \floor{\log\log(n-1) }$. Hence the upper bound in Corollary \ref{corollarySubdivisionChromaticNumber} is tight up to a constant factor.


\subsection{Hypercube}
\label{sectionHypercube}


\begin{definition}
\label{definitionHypercube}
For a positive integer $d$, the {\em $d$-dimensional hypercube} $Q_d$ is the graph with $2^d$ vertices where each vertex $v$ corresponds to a distinct $d$-bit binary string $g(v)$ such that two vertices $u, v \in V(Q_d)$ are adjacent if and only if $g(u)$ differs from $g(v)$ at exactly one bit position. Let $g_i(v)$ denote the $i$-th bit from right in $g(v)$, where $i \in [d]$. The number of ones in $g(v)$ is called the {\em hamming weight} of $v$ and is denoted by $h(v)$.
\end{definition}

\begin{observation}
\label{observationHypercubeNonadjacentEdges}
Let $a,b,c,$ and $d$ be four distinct vertices in the hypercube $Q_d$ with $\{a,b\}, \{c,d\} \in E(Q_d)$ such that $g(a)$ and $g(b)$ differ only in the $i$-th bit position from right and $g(c)$ and $g(d)$ differ only in the $j$-th position from right. Then there exists some $k \in [d] \setminus \{i,j\}$ such that $g_k(a)$ ($=g_k(b)$) differs from $g_k(c)$ ($=g_k(d)$).
\end{observation} 
\begin{proof}
Assume for contradiction that, for every $k \in [d] \setminus \{i,j\}$, $g_k(a) = g_k(b) = g_k(c) = g_k(d)$. if $i=j$ then there can be only $2$ distinct binary strings among $\{g(a), g(b), g(c), g(d) \}$. If $i \neq j$, then there can only be $3$ distinct binary strings among $\{g(a), g(b), g(c), g(d) \}$ since the $i$-th and $j$-th bit positions from right cannot simultaneously be $1 - g_i(c)$ and $1-g_j(a)$ respectively for any of the $4$ strings in the set.  This contradicts the distinctness of $a,b,c,$ and $d$. 
\end{proof}

\begin{theorem}
\label{theoremHypercube}
For the $d$-dimensional hypercube $Q_d$, 
$$ \frac{1}{2} \floor{\log\log(d-1)} \leq 
		\boxli(Q_d) \leq
		\floor{ \log\log d + \frac{1}{2} \log\log\log d + \log(\sqrt{2}\pi) + o(1) }
.$$ 
\end{theorem}
\begin{proof}
Let $H$ be the subgraph of $Q_d$ induced on the vertex set $V(H) = \{v \in V(Q_d)~:~h(v) \in \{1,2\}\}$. Observe that $H$ is isomorphic to $K_d^{1/2}$ and therefore, by Theorem \ref{theoremKnHalf}, $\boxli(H) \geq \floor{ \log\log (d-1) }$. Hence the lower bound follows from  by Observation \ref{observationMonotonicity}.

Next we show the upper bound by using $3$-suitable permutations of the bit positions. Let $\E = \{ \sigma_1, \ldots , \sigma_r \}$ be a smallest $3$-suitable family of permutations of $[d]$. From \cite{scramble}, we know that $r \leq \floor{ \log \log d + \frac{1}{2}\log \log \log d + \log(\sqrt{2} \boxli) + o(1)  }$. For a permutation $\sigma \in \E$ and a pair $u,v \in V(Q_d)$, let $i_{\sigma}(u,v)$ denote the largest value of $\sigma(i)$ (over all $i \in [d]$) for which $g_{\sigma(i)}(u) \neq  g_{\sigma(i)}(v)$, i.e., the right most bit position where $u$ and $v$ differ if the bit positions are permuted according to $\sigma$. From $\E$, we construct a family of permutations $\F = \{\tau_1, \ldots , \tau_r\}$ that is pairwise suitable for $Q_d$. The permutation $\tau_j$ is constructed by first permuting the bit positions of all the binary strings according to $\sigma_j$ and then reading out the vertices in the right to left lexicographic order of the bit strings. That is, for $u, v \in V(Q_d)$, $u \prec_{\tau_j} v$ if $g_i(u) < g_i(v)$, where $i = i_{\sigma_j}(u,v)$.


In order to show that $\mathcal{F}$ is a pairwise suitable family of permutations for $Q_d$, consider two disjoint edges $\{a,b\}$, $\{c,d\}$ in $Q_d$ such that $g(a)$ and $g(b)$ differ only in the $l$-th position from right and $g(c)$ and $g(d)$ differ only in the $m$-th position from right. Then, from Observation \ref{observationHypercubeNonadjacentEdges}, we know that there exists a $k \in [d] \setminus \{l,m\}$ such that $g_k(a)$ ($=g_k(b)$) differs from $g_k(c)$ ($=g_k(d)$). Since $\E$ is a $3$-suitable family of permutations for $[d]$, there exists a $\sigma_s \in \E$ such that $ \{l,m\} \prec_{\sigma_s} k$. That is, $\sigma_s(l) < \sigma_s(k)$ and $\sigma_s(m) < \sigma_s(k)$. 
Hence, $i_{\sigma_s}(u,v) \geq \sigma_s(k), \, u \in \{a,b\}, v \in \{c,d\}$. It then follows from the definition of $\tau_s$ that either $\{a,b\}  \prec_{\tau_s} \{c,d\}$ or $\{c,d\} \prec_{\tau_s} \{a,b\}$.

\end{proof}



\section{Lower bounds}
\label{sectionLowerBound}

The tightness of many of the upper bounds we showed in the previous section relies on the lower bounds we derive in this section. First, we show that if a graph contains a uniform bipartite subgraph, then it needs a large separation dimension. This immediately gives a lower bound on separation dimension for complete bipartite graphs and hence a lower bound for every graph $G$ in terms $\omega(G)$. The same is used to obtain a lower bound on the separation dimension for random graphs of all density. Finally, it is used as a critical ingredient in proving a lower bound on the separation dimension for complete $r$-uniform hypergraphs. Before we close this section we give a lower bound on the separation dimension of $K_n^{1/2}$ using Erd\H{o}s-Szekeres Theorem and a lower bound on the poset dimension of canonical interval orders.


\subsection{Uniform bipartitions}
\label{sectionLowerBoundBipartition}

\begin{theorem}
\label{theoremBoxliLowerBound}
For a graph $G$, let $V_1, V_2 \subsetneq V(G)$ such that $V_1 \cap V_2 = \emptyset$. If  there exists an edge between every $s_1$-subset of $V_1$ and every $s_2$-subset of $V_2$, then $\boxli(G) \geq \min \left\{ \log \frac{|V_1|}{s_1}, \log \frac{|V_2|}{s_2} \right\}$.  
\end{theorem}
\begin{proof}
Let $\F$ be a family of permutations of $V(G)$ that is pairwise suitable for $G$. Let $r = |\F|$. We claim that, for any $\sigma \in \F$, there always exists an $S_1 \subseteq V_1$ and an $S_2 \subseteq V_2$ such that $|S_1| \geq \ceil{ |V_1|/2 }, |S_2| \geq \ceil{ |V_2|/2 }$ and $S_1 \prec_{\sigma} S_2$ or $S_2 \prec_{\sigma} S_1$. To see this, scan $V(G)$ in the order of $\sigma$ till we see $\ceil{|V_1|/2 }$ elements from $V_1$ or $\ceil{|V_2|/2}$ elements of $V_2$, which ever happens earlier. In the former case the first $\ceil {|V_1|/2 }$ elements of $V_1$ precede at  least $\ceil{ |V_2|/2 }$ elements of $V_2$ and in the latter case the first $\ceil{|V_2|/2}$ elements of $V_2$ precede at least $\ceil{|V_1|/2}$ elements of $V_1$. Extending this claim recursively to all permutations in $\F$, we see that there always exist a $T_1 \subseteq V_1$ and a $T_2 \subseteq V_2$ such that $|T_1| \geq |V_1|/2^r, |T_2| \geq |V_2|/2^r$ and $\forall \sigma \in \F$, either $T_1 \prec_{\sigma} T_2$ or $T_2 \prec_{\sigma} T_1$. We now claim that either $|T_1| \leq s_1$ or  $|T_2| \leq s_2$. Suppose, for contradiction, $|T_1| \geq s_1+1$ and  $|T_2| \geq s_2+1$. Then by the statement of the theorem, there exists an edge $e = \{v_1, v_2\}$ of $G$ such that $v_1 \in T_1$ and $v_2 \in T_2$ and a second edge $f$ between $T_1 \setminus \{v_1\}$ and $T_2 \setminus \{v_2\}$. Since $T_1$ and $T_2$ are separated in every permutation of $\F$, no permutation in $\F$ separates the disjoint edges $e$ and $f$ between $T_1$ and $T_2$. This contradicts the fact that $\F$ is a pairwise suitable family for $G$. Hence, either $|V_1| / 2^r \leq |T_1| \leq s_1$ or $|V_2|/2^r \leq |T_2| \leq s_2$ or both. That is, $r \geq \min \left\{ \log \frac{|V_1|}{s_1}, \log \frac{|V_2|}{s_2} \right\}$. 
\end{proof}

The next two corollaries are immediate. 

\begin{corollary}
\label{corollaryCompleteBipartiteLowerBound}
For a complete bipartite graph $K_{m,n}$ with $m \leq n$, $\boxli(K_{m,n}) \geq \log(m)$. 
\end{corollary}
\begin{corollary}
\label{corollaryBoxliOmega}
For a graph $G$, 
$$\boxli(G) \geq \log \floor{\frac{\omega}{2}},$$
where $\omega$ is the size of a largest clique in $G$.
\end{corollary}


\subsection{Random graphs}
\label{sectionRandomGraphs}


\begin{definition}[Erd\H{o}s-R\'{e}nyi model]
$\G(n,p)$, $n \in \N$ and $0 \leq p \leq 1$, is the discrete probability space of all simple undirected graphs $G$ on $n$ vertices with each pair of vertices of $G$ being joined by an edge with a probability $p$ independent of the choice for every other pair of vertices. 
\end{definition}

\begin{definition}
A property $P$ is said to hold for $\G(n,p)$ {\em asymptotically almost surely (a.a.s)} if the probability that $P$ holds for $G \in \G(n,p)$ tends to $1$ as $n$ tends to $\infty$.
\end{definition}

%
%
%


\begin{theorem}
\label{theoremBoxliLowerBoundRandom}
For $G \in \G(n,p(n))$ 
$$\boxli(G) \geq \log(np(n)) - \log\log(np(n)) - 2.5 \mbox{ a.a.s}.$$
\end{theorem}
\begin{proof}

If $np(n) \leq e^{e/4}$, then  $\log(np(n)) - \log\log(np(n)) - 2.5 \leq 0$, and hence the statement is trivially true. So we can assume that $p(n) > e^{e/4} / n$. 

Let $s(n) = 2 \ln (np(n)) / p(n)$. Since $p(n) > e^{e/4} / n$ by assumption, $\ln (np(n)) > e/4$ and hence if $\lim_{n \tends \infty} p(n) = 0$, we get $\lim_{n \tends \infty} s(n) = \infty$. Otherwise, that is when $\liminf_{n \tends \infty} p(n) > 0$, we have $s(n) \geq 2 \ln(np(n)) / 1$ which tends to $\infty$ as $n \tends \infty$. Hence in every case $\lim_{n \tends \infty} s(n) = \infty$.

Let $V(G) = V_1 \uplus V_2$ be a balanced partition of $V(G)$, i.e., $V_1 \cap V_2 = \emptyset$ and $|V_1|, |V_2| \geq \floor{ n/2  }$. $S_1 \subseteq V_1$ and $S_2 \subseteq V_2$ be such that $|S_1| = |S_2| = s(n)$. The probability that there is no edge in $G$ between $S_1$ and $S_2$ is $(1-p(n))^{s(n)^2} \leq \exp(-p(n)s(n)^2)$. Hence the probability $q(n)$ that there exists an $s(n)$-sized set from $V_1$ and one $s(n)$-sized set from $V_2$ with no edge between them is bounded above by ${n/2 \choose s(n)}^2 \exp(-p(n)s(n)^2)$. Hence using the bound ${n \choose k} \leq (ne/k)^k$, we get
\begin{eqnarray*}
q(n) 	& \leq 	& 	\left( \frac{ne}{2s(n)} \right)^{2s(n)} \exp( -p(n)s(n)^2) \\ 
		& = 	&	\exp \left(2s(n) \ln \left( \frac{ne}{2s(n)} \right) - p(n)s(n)^2 \right) \\
		& = 	&	\exp \left(s(n) \left( 2\ln \left( \frac{np(n)e}{4\ln(np(n))}\right)  - 2 \ln(np(n)) \right) \right) \\
		& = 	&	\exp \left(s(n) \left( 2\ln \frac{e}{4}  - 2 \ln\ln(np(n)) \right) \right) \\
		& = 	&	\exp \left(-2s(n)\left(\ln\ln(np(n)) - \ln \frac{e}{4} \right) \right) \\
\end{eqnarray*}

Since $p(n) > e^{e/4} / n$, $\ln\ln(np(n)) > \ln(e/4)$ and since $\lim_{n \tends \infty} s(n) = \infty$, we conclude that  $\lim_{n \tends \infty} q(n) = 0$.

With probability $1 - q(n)$, every pair of subsets from $V_1 \times V_2$ each of size $s(n)$ has at least one edge between them. So by Theorem \ref{theoremBoxliLowerBound}, $\boxli(G) \geq \log \floor{n/2s(n)} \geq \log(np(n)) - \log\log (np(n)) - 2.5$ with probability $1 - q(n)$. Hence the theorem.
\end{proof}

Note that the expected average degree of a graph in $\G(n,p)$ is $\mathbb{E}_p[\bar{d}] = (n-1)p$. And hence the above bound can be written as $\log \mathbb{E}_p[\bar{d}] - \log\log \mathbb{E}_p[\bar{d}] - 2.5$.


\subsection{Hypergraphs}
\label{sectionLowerBoundHypergraphs}

Now we illustrate one method of extending the above lower bounding technique from graphs to hypergraphs.
Let $K_n^r$ denote the complete $r$-uniform hypergraph on $n$ vertices. We show that the upper bound of $\order{4^r \sqrt{r} \log n}$ obtained for $K_n^r$ from Theorem \ref{theoremHypergraphSizeUpperbound} is tight up to a factor of $r$. The lower bound argument below is motivated by an argument used by Radhakrishnan to prove a lower bound on the size of a family of scrambling permutations \cite{radhakrishnan2003note}. 

\begin{theorem}
\label{theoremHypergraphSizeLowerbound}
Let $K_n^r$ denote the complete $r$-uniform hypergraph on $n$ vertices with $r > 2$. Then 
$$ c_1 \frac{4^r}{\sqrt{r-2}} \log n \leq \boxli(K_n^r) \leq c_2 4^r \sqrt{r} \log n,$$ for $n$ sufficiently larger than $r$ and where $c_1 = \frac{1}{2^7}$ and  $c_2 =  \frac{e\ln2}{\pi\sqrt{2}} < \frac{1}{2}$.
\end{theorem}
\begin{proof} The upper bound follows from Theorem \ref{theoremHypergraphSizeUpperbound} and so it suffices to prove the lower bound.

\def\S{\mathcal{S}}

Let $\F$ be a family of pairwise suitable permutations for $K_n^r$. Let $\S$ be a maximal family of $(r-2)$-sized subsets of $[2r-4]$ such that if $S \in \S$, then $[2r-4] \setminus S \notin \S$. Hence $|\S| = \frac{1}{2} {2r-4 \choose r-2} \geq 2^{-6} 4^r / \sqrt{r-2}$ (using the fact that $\sqrt{k} {2k \choose k} \geq 2^{2k-1}$). Notice that for any permutation $\sigma \in \F$, if $S \in \S$ and $[2r-4] \setminus S$ are separated in $\sigma$ then no other $S' \in \S$ and $[2r-4] \setminus S'$ are separated in $\sigma$. Hence we partition $\F$ into $|\S|$ (disjoint) sub-families $\{\F_S\}_{S \in \S}$ such that $\sigma \in \F_S$ if and only if $\sigma$ separates $S$ and $[2r-4] \setminus S$. We claim that each $\F_S$ is pairwise suitable for the complete graph on the vertex set $\{2r-3, \ldots, n\}$, i.e, for any distinct $a, b, c, d \in \{2r-3, \ldots, n\}$ there exists some $\sigma \in \F_S$ which separates $\{a, b\}$ from  $\{c, d\}$. This is because the permutation $\sigma \in \F$ which separates the $r$-sets $S \cup \{a,b\}$ from $([2r-4] \setminus S) \cup \{c, d\}$ lies in $\F_S$. Hence by Corollary \ref{corollaryBoxliOmega}, we have $|\F_S| \geq \log \floor{(n - 2r +4)/2}$. Since $\F = \biguplus_{S \in \S} \F_S$, we have $|\F| \geq |\S||\F_S| \geq 2^{-6} \frac{4^r}{\sqrt{r-2}} \log \floor{ (n-2r+4)/2 }$ which is at least $2^{-7} \frac{4^r}{\sqrt{r-2}} \log n$ for $n$ sufficiently larger than $r$.
\end{proof}


\subsection{Fully subdivided clique} 
\label{sectionSubdividedClique}

It easily follows from Corollary \ref{corollarySubdivisionChromaticNumber} that $\boxli(K_n^{1/2}) \in O(\log \log n)$. In this section we prove that $\boxli(K_n^{1/2}) \geq \frac{1}{2} \log \log (n-1)$, showing the near tightness of that upper bound. We give a brief outline of the proof below. (Definitions of the new terms are given before the formal proof.) 

First, we use Erd\H{o}s-Szekeres Theorem \cite{ErdosSzekeres} to argue that for any family $\F$ of permutations of $V(K_n^{1/2})$, with $|\F| < \frac{1}{2} \log \log n$, a subset $V'$ of original vertices of $K_n^{1/2}$, with $n' = |V'| \approx 2^{\sqrt{\log n}}$, is ordered essentially in the same way by every permutation in $\F$. Since the ordering of the vertices in $V'$ are fixed, the only way for $\F$ to realise pairwise suitability among the edges in the subdivided paths between vertices in $V'$ is to find suitable positions for the new vertices (those introduced by subdivisions) inside the fixed order of $V'$. We then show that this amounts to constructing a realiser for the canonical open interval order $(C_{n'}, \lhd)$ and hence $|\F|$, in this case, is lower bounded by the poset dimension of $(C_{n'}, \lhd)$ which is known to be at least $\log \log (n'-1) = \frac{1}{2} \log \log (n-1)$. 

\begin{definition}[Canonical open interval order]
\label{definitionCanonicalOpenInterval}
For a positive integer $n$, let $C_n = \{(a,b) : a, b \in [n], a < b \}$ be the collection of all the ${n \choose 2}$ open intervals which have their endpoints  in $[n]$. Then $(C_n,\lhd)$, the interval order corresponding to the collection $C_n$, is called the {\em canonical open interval order}. 
\end{definition} 

Usually the canonical interval order is defined over closed intervals.  For a positive integer $n$, let $I_{n} = \{[a,b]: a, b \in [n], a \leq b \}$ be the collection of all the ${n+1 \choose 2}$ closed intervals which have their endpoints in $[n]$. The poset $(I_n,\lhd')$, where $[i, j] \lhd' [k,l] \iff j < k$ is called the {\em canonical (closed) interval order} in literature. It is easy to see that $f: (C_n, \lhd) \into (I_{n-1}, \lhd')$, with $f((i,j)) = [i, j-1]$ is an isomorphism. It is well known that the dimension of $(I_{n-1}, \lhd')$ and hence $(C_n, \lhd)$ is at most $\log\log (n-1) + (\frac{1}{2} + o(1))\log\log\log (n-1)$. We state the lower bound below for later reference.

\begin{theorem}[F\"{u}redi, Hajnal, R\"{o}dl, Trotter \cite{furedi1991interval}]
\label{theoremIdimCanonicalLowerBound}
$$dim(C_n) \geq \log\log(n-1),$$
\end{theorem}

\begin{theorem}
\label{theoremKnHalf}
Let $K_n^{1/2}$ denote the graph obtained by fully subdividing $K_n$. Then, 
$$ \frac{1}{2} \floor{\log\log(n-1)} \leq \boxli(K_n^{1/2}) \leq (1 + o(1))\log\log (n-1).$$ 
\end{theorem}
\begin{proof}
The upper bound follows from Corollary \ref{corollarySubdivisionChromaticNumber}. So it suffices to show the lower bound. 

Let $v_1, \ldots , v_n$ denote the \textit{original vertices} (the vertices of degree $n-1$) in $K_n^{1/2}$ and let $u_{ij}$, $i, j \in [n], i < j$, denote the new vertex of degree $2$ introduced when the edge $\{i,j\}$ of $K_n$ was subdivided. Let $\F$ be a family of permutations that is pairwise suitable for $K_n^{1/2}$ such that $|\F| = r = \boxli(K_n^{1/2})$. For convenience, let us assume that $n$ is exactly one more than a power of power of $2$, i.e., $\log\log (n-1) \in \N$. The floor in the lower bound gives the necessary correction otherwise when we bring $n$ down to the largest such number below $n$. Let $p=(n-1)^{1/2^r} + 1$. 

By Erd\H{o}s-Szekeres Theorem \cite{ErdosSzekeres},  we know that if $\tau$ and $\tau'$ are two permutations of $[n^2 + 1]$, then there exists some $X \subseteq [n^2 + 1]$ with $|X|=n+1$ such that the permutations $\tau$ and $\tau'$ when restricted to $X$ are the same or reverse of each other. By repetitive application of this argument, we can see that there exists a set $X$ of $p$ original vertices of $K_n^{1/2}$ such that, for each $\sigma, \sigma' \in \F$, the permutation of $X$ obtained by restricting $\sigma$ to $X$ is the same or reverse of the permutation obtained by restricting $\sigma'$ to $X$. Without loss of generality, let $X = \{v_1, \ldots , v_p\}$ such that, for each $\sigma \in \F$, either $v_1 \prec_{\sigma} \cdots \prec_{\sigma} v_p$ or $v_p \prec_{\sigma} \cdots \prec_{\sigma} v_1$. Now we ``massage'' $\F$ to give it two nice properties without changing its cardinality or sacrificing its pairwise suitability for $K_n^{1/2}$.

Note that if a family of permutations is pairwise suitable for a graph then the family retains this property even if any of the permutations in the family is reversed. Hence we can assume the following property without loss of generality.

\setcounter{property}{0}
\begin{property}
\label{property1KnHalf}
$v_1 \prec_{\sigma} \cdots \prec_{\sigma} v_p, \forall \sigma \in \F$. 
\end{property}

Consider any $i,j \in [p], i < j$. For each $\sigma \in \F$, it is safe to assume that $v_i \prec_{\sigma} u_{ij} \prec_{\sigma} v_j$. Otherwise, we can modify the permutation $\sigma$ such that $\F$ is still a pairwise suitable family of permutations for $K_n^{1/2}$. To demonstrate this, suppose $v_i  \prec_{\sigma} v_j \prec_{\sigma} u_{ij}$. Then, we modify $\sigma$ such that $u_{ij}$ is the immediate predecessor of $v_j$. It is easy to verify that, for each pair of disjoint edges $e,f \in E(K_n^{1/2})$, if $e \prec_{\sigma} f$ or  $f \prec_{\sigma} e$ then the same holds in the modified $\sigma$ too. Similarly, if $u_{ij} \prec_{\sigma} v_i  \prec_{\sigma} v_j$ then we modify $\sigma$ such that $u_{ij}$ is the immediate successor of $v_i$. Hence we can assume the next property also without loss in generality.

\begin{property}
\label{property2KnHalf}
$v_i \prec_{\sigma} u_{ij} \prec_{\sigma} v_j, \forall i, j \in [p], i <j,~ \forall \sigma \in \F$.
\end{property}

These two properties ensure that for any two open intervals $(i,j)$ and $(k,l)$ in $C_p$ if $(i,j) \lhd (k,l)$ then $u_{ij} \prec_{\sigma} u_{kl}, \forall \sigma \in \F$. In the other case, i.e., when $(i,j) \cap (k,l) \neq \emptyset$, we make the following claim.

\begin{claim}
\label{claim1KnHalf}
Let $i,j,k,l \in [p]$ such that $(i,j) \cap (k,l) \neq \emptyset$. Then there exist $\sigma_a, \sigma_b \in \F$ such that $u_{ij} \prec_{\sigma_a} u_{kl}$ and $u_{kl} \prec_{\sigma_b} u_{ij}$.
\end{claim}


Since $(i,j) \cap (k,l) \neq \emptyset$, we have $k < j$ and $i < l$. Hence by Property \ref{property1KnHalf}, $\forall \sigma \in \F$, $v_k \prec_{\sigma} v_j$ and $v_i \prec_{\sigma} v_l$. Now we prove the claim by contradiction. If  $u_{ij} \prec_{\sigma} u_{kl}$ for every $\sigma \in \F$ then, together with the fact that $v_k \prec_{\sigma} v_j, \forall \sigma \in \F$, we see that no $\sigma \in \F$ can separate the edges $\{v_j, u_{ij}\}$ and $\{v_k, u_{kl}\}$. But this contradicts the fact that $\F$ is a pairwise suitable family of permutations for $K_n^{1/2}$.  Similarly if $u_{kl} \prec_{\sigma} u_{ij}$ for every $\sigma \in \F$ then, together with the fact that $v_i \prec_{\sigma} v_l, \forall \sigma \in \F$, we see that no $\sigma \in \F$ can separate $\{v_i, u_{ij}\}$ and $\{v_l, u_{kl}\}$. But this too contradicts the pairwise suitability of $\F$. Thus we prove Claim \ref{claim1KnHalf}.

With these two properties and the claim above, we are ready to prove the following claim.

\begin{claim}
\label{claim2KnHalf}
$|\F| \geq \idim((C_p, \lhd))$.
\end{claim}

For every $\sigma \in \F$, construct a total order $L_{\sigma}$ of $C_p$ such that $(i,j) \lhd (k,l) \in L_{\sigma} \iff u_{ij} \prec_{\sigma} u_{kl}$. By Property \ref{property1KnHalf} and Property \ref{property2KnHalf}, $L_{\sigma}$ is a linear extension of $(C_p, \lhd)$. Further, Claim \ref{claim1KnHalf} ensures that $\mathcal{R} = \{L_{\sigma}\}_{\sigma \in \F}$ is a realiser of $(C_p, \lhd)$. Hence $|\F| = |\mathcal{R}| \geq \idim((C_p, \lhd))$.


Now we are ready to show the final claim which settles the lower bound.

\begin{claim}
\label{claim3KnHalf}
$|\F| \geq \frac{1}{2}\log \log (n-1)$. 
\end{claim}

Suppose for contradiction that $|\F| = r < \frac{1}{2}\log \log (n-1)$. Then, by Claim \ref{claim2KnHalf}, $r \geq \idim((C_p, \lhd))$ where $p = (n-1)^{1/2^r} + 1 > 2^{\sqrt{\log (n-1)}} + 1$. But then, by Theorem \ref{theoremIdimCanonicalLowerBound}, we have $r \geq \log \log (p-1) > \log \log (2^{\sqrt{\log (n-1)}}) = \frac{1}{2}\log \log (n-1)$ which contradicts our starting assumption. 
\end{proof}




\section{Discussion and open problems}
\label{sectionOpenProblems}
For a graph $G$, we have given upper bounds for $\boxli(G)$ exclusively in terms of $|V(G)|$, $\Delta(G)$, $\tw(G)$, $\chi_a(G)$ and $\chi_s(G)$. Hence it is natural to ask if a lower bound can be given for $\boxli(G)$ exclusively in terms of any of these parameters. The answer turns out to be negative at least for the first three. An empty graph $E_n$ on $n$ vertices has $\boxli(E_n) = 0$.  The star graph $S_{n-1}$ on $n-1$ leaves has $\Delta(S_n) = n-1$, but $\boxli(S_n) = 0$. The $n \times n$ square grid $G$ on the plane has a treewidth $n$ but a bounded $\boxli(G)$ since it is planar. In fact $\boxli(G) = 2$ since a plane drawing of $G$ as an axis-parallel grid is a $2$-box representation of $L(G)$. As for $\chi_a(G)$, and hence $\chi_s(G)$, we cannot hope to get an exclusive lower bound for $\boxli(G)$ of a larger order than $\log \log \chi_a(G)$. This is because if $G$ is the graph obtained by replacing every edge of $K_n$ with $n-1$ parallel paths of length $2$, then it is easy to see that (by two applications of pigeonhole principle) $\chi_a(G) \geq n$ \cite{kostochka1976note}. But since $G$ is $2$-degenerate we know that $\boxli(G) \in \order{\log \log |G|}$ and $|G| \leq n^3$.

In view of the above, it is natural to ask what other graph parameters, apart from $\omega(G)$, have a potential to give an exclusive lower bound for $\boxli(G)$. Two parameters that we have tried are the Hadwiger number $\eta(G)$ and chromatic number $\chi(G)$. The Hadwiger number of a graph is the size of a largest clique minor in $G$. Note that $\tw(G) + 1 \geq \eta(G)$ and if the Hadwiger conjecture is true, then $\eta(G) \geq \chi(G) \geq \omega(G)$. The possibility of getting an exclusive lower bound for $\boxli(G)$ in terms of $\eta(G)$ is ruled out because the double $n \times n$ square grid $G$, i.e., the graph obtained by taking two identical $n \times n$ square grids and connecting the identical nodes with an edge, has $\eta(G) \geq n$ \cite{chandran2007hadwiger} but $\boxli(G) \leq 3$. Here again an axis parallel $3$-dimensional drawing of $G$ is a $3$-box representation. We have shown that $\boxli(G) \geq \log \floor{\omega(G)/2}$. But a similar bound in terms of $\chi(G)$ could not be arrived at and hence we pose the following question.

\begin{openproblem}
For any graph $G$, is $\boxli(G) \geq \log \chi(G) - c$, for some constant $c$?
\end{openproblem}

The answer is positive for graphs like perfect graphs where $\chi(G) = \omega(G)$. Notice that we cannot have an upper bound for $\boxli(G)$ exclusively in terms of $\chi(G)$ since the complete bipartite graph $K_{n,n}$ has $\boxli(K_{n,n}) \geq \log n$, but $\chi(K_{n,n}) = 2$.

Among the upper bounds that are obtained in this paper, for which we do not know any reasonable tightness, the one based on $\Delta(G)$ (Theorem \ref{theoremBoxliDelta}) is the one that has engaged us the most. We saw that $\boxli(G) \leq 2^{9 \ilog \Delta(G)} \Delta(G)$. For a chordal graph $G$, by Theorem \ref{theoremBoxliTreewidth}, we have $\boxli(G) \in O(\log \Delta(G))$,  since $\omega(G) -1  \leq \Delta(G)$. For a graph $G$ with $\Delta(G)$ of order at least $\log n$, by Theorem \ref{theoremBoxliSize} (on $|V(G)|$), we have $\boxli(G) \in \order{\Delta(G)}$. On the other hand, the  examples of sparse graphs that we have studied, together with the monotonicity of $\boxli(G)$ tempts us to make the following conjecture.

\begin{conjecture}
For a graph $G$ with maximum degree $\Delta(G)$, $\boxli(G) \in \order{\Delta(G)}$. 
\end{conjecture}

%


Since $\boxli(G)$ is the boxicity of the line graph of $G$, it is interesting to see how it is related to boxicity of $G$ itself. But unlike separation dimension, boxicity is not a monotone parameter. For example the boxicity of $K_n$ is $1$, but deleting a perfect matching from $K_n$, if $n$ is even, blows up its boxicity to $n/2$. Yet we couldn't find any graph $G$ such that $\boxi(G) > 2^{\boxli(G)}$. Hence we are curious about the following question.

\begin{openproblem}
Does there exist a function $f : \mathbb{N} \into \mathbb{N}$ such that $\boxi(G) \leq f(\boxli(G))$?
\end{openproblem}

Note that the analogous question for $\boxlistar(G)$ has an affirmative answer. If there exists a vertex $v$ of degree $d$ in $G$, then any $3$-mixing family of permutations of $V(G)$ should contain at least $\log d$ different permutations because any single permutation will leave $\ceil{d/2}$ neighbours of $v$ on the same side of $v$. Hence $\log \Delta(G) \leq \boxlistar(G)$. From \cite{DiptAdiga}, we know that $\boxi(G) \in \order{\Delta(G) \log^2 \Delta(G)}$ and hence $\boxi(G) \in \order{2^{\boxlistar(G)} (\boxlistar(G))^2}$.

The upper and lower bounds for $\boxli(K_n^r)$, given by Theorem \ref{theoremHypergraphSizeLowerbound} differ by a factor of $r$. Estimating the exact order growth of $\boxli(K_n^r)$ will be a challenging question. A similar gap of $r^2$ is present in the upper and lower bounds for the size of a smallest family of completely $r$-scrambling permutations of $[n]$ (See \cite{radhakrishnan2003note}).

\begin{openproblem}
What is the exact order growth of $\boxli(K_n^r)$?
\end{openproblem}

Another interesting direction of enquiry is to find out the maximum number of hyperedges (edges) possible in a hypergraph (graph) $H$ on $n$ vertices with $\boxli(H) \leq k$. Such an extremal hypergraph $H$, with $\boxli(H) \leq 0$, is seen to be a maximum sized intersecting family of subsets of $[n]$. A similar question for order dimension of a graph has been studied \cite{agnarsson1999maximum,agnarsson2002extremal} and has found applications in ring theory. We can also ask a three dimensional analogue of the question answered by Schnyder's theorem in two dimensions. Given a collection $P$ of non parallel planes in $\R^3$, can we embed a graph $G$ in $\R^3$ so that every pair of disjoint edges is separated by a plane parallel to one in $P$. Then $|P|$ has  to be at least $\boxli(G)$ for this to be possible. This is because the permutations induced by projecting such an embedding onto the normals to the planes in $P$ gives a pairwise suitable family of permutations of $G$ of size $|P|$. Can $|P|$ be upper bounded by a function of $\boxli(G)$?

\bibliographystyle{plain}
\bibliography{mathewref}

\end{document}